%% file: KMKostantPaper.tex
\documentclass[a4paper]{article}

\input{arXivSetup/generalSetup}

\input{arXivSetup/theorems}
\input{arXivSetup/commands}

\newcommand{\Addresses}{
		\bigskip
		\footnotesize
		
		\noindent Paul Zellhofer, \textsc{Mathematisches Seminar, Christian-Albrechts-Universität zu Kiel,
			Heinrich-Hecht-Platz 6, 24118 Kiel} \\
		E-mail address: \texttt{zellhofer@math.uni-kiel.de}
		
		\medskip
		
		\noindent Ralf Köhl, \textsc{Mathematisches Seminar, Christian-Albrechts-Universität zu Kiel,
			Heinrich-Hecht-Platz 6, 24118 Kiel} \\
		E-mail address: \texttt{koehl@math.uni-kiel.de}
}

\title{On convexity and the Iwasawa decomposition of split real and complex Kac-Moody groups}
\date{January 27, 2024}
\author{Paul Zellhofer}
\author{Ralf Köhl}
\affil{Christian-Albrechts-Universität zu Kiel}

\begin{document}
	\maketitle
	
\begin{abstract}
	\noindent We prove an analogue of Kostant's convexity theorem for split real and complex Kac-Moody groups associated to free and cofree root data. The result can be seen as a first step towards describing the multiplication map in a Kac-Moody group in terms of Iwasawa coordinates. Our method involves a detailed analysis of the geometry of Weyl group orbits in the Cartan subalgebra of a real Kac-Moody algebra. It provides an alternative proof of Kostant convexity for semisimple Lie groups and also generalizes a linear analogue of Kostant's theorem for Kac-Moody algebras that has been established by Kac and Peterson in 1984.
\end{abstract}
	
\input{Sections/KMKostantPaper_Intro}
\input{Sections/KMKostantPaper_KMAlgebras}
\input{Sections/KMKostantPaper_ConvexHulls}
\input{Sections/KMKostantPaper_KMGroups}
\input{Sections/KMKostantPaper_MainTheorem}
	
\bibliography{KMKostantPaper_References.bib}
\bibliographystyle{alpha}

\Addresses
	
\end{document}

%% file: arXivSetup/generalSetup.tex
\usepackage{a4wide} %nur für jetzt, durch 'geometry' ersetzen!
\usepackage[T1]{fontenc} %Font-Kodierung z.B. für Umlaute
\usepackage[english]{babel} %Spracheinstellung
\usepackage[affil-it]{authblk}

\usepackage{amsmath,amssymb,mathtools,mathdots} %Mathe-Pakete
\usepackage{graphicx} %Einfügen von Grafiken
	\usepackage{subcaption}
\usepackage{tikz-cd} \usetikzlibrary{babel} %kommutative Diagramme
\usepackage{enumitem}

\usepackage[colorlinks=false,naturalnames=true,plainpages=false,pdfpagelabels=true]{hyperref} % Verweise

\numberwithin{equation}{section}

%% file: arXivSetup/theorems.tex
\usepackage{amsthm}
\usepackage{thmtools}

\declaretheorem[
	name=Theorem,
	numberwithin=section
	]{thm}
\declaretheorem[
	name=Lemma,
	sibling=thm,
	]{lem}
\declaretheorem[
	name=Proposition,
	sibling=thm,
	]{prop}
\declaretheorem[
	name=Corollary,
	sibling=thm,
	]{cor}

\declaretheorem[
	name=Remark,
	style=remark,
	sibling=thm,
	]{rmk}	
\declaretheorem[
	name=Example,
	style=remark,
	sibling=thm,
	]{exam}

\declaretheorem[
	name=Assumption,
	style=definition,
	sibling=thm,
	]{assum}
\declaretheorem[
	name=Main Theorem,
	numbered=no
]{mthm}

%% file: arXivSetup/commands.tex
\newcommand{\N}{\mathbb{N}}
\newcommand*{\Z}{\mathbb{Z}} 

\newcommand{\R}{\mathbb{R}}
\newcommand*{\C}{\mathbb{C}}
\newcommand*{\K}{\mathbb{K}}
\newcommand*{\A}{\mathbb{A}}

\newcommand{\id}{\operatorname{id}} %häufige Ausdrücke
\newcommand{\ad}{\operatorname{ad}}
\newcommand{\Ad}{\operatorname{Ad}}

\newcommand{\conv}{\operatorname{conv}}

\newcommand{\SL}{\operatorname{SL}} %häufige Matrixgruppen
\newcommand{\SO}{\operatorname{SO}}
\newcommand{\GL}{\operatorname{GL}}

\newcommand{\g}{\mathfrak{g}} %häufige Fraktur-Buchstaben
\newcommand{\h}{\mathfrak{h}}
\renewcommand{\k}{\mathfrak{k}}
\renewcommand{\a}{\mathfrak{a}}

\newcommand{\p}{\mathfrak{p}}

\newcommand{\mc}[1]{\mathcal{#1}}
\newcommand{\mf}[1]{\mathfrak{#1}}
\newcommand{\op}[1]{\operatorname{#1}}

%% file: Sections/KMKostantPaper_Intro.tex
\section*{Introduction}

\textit{Kostant convexity} is a fundamental property of semisimple Lie groups that has been established by B. Kostant in \cite{Kostant}. The origins of Kostant's theorem are related to the following problem. Let $G$ be a connected semisimple real Lie group with Lie algebra $\g$ and $\g=\k\oplus\p$ a Cartan decomposition of $\g$, where $\k$ is the fixed-point set of a Cartan involution $\theta:\g\to\g$ and $\p$ is the $(-1)$-eigenspace of $\theta$. If $\a$ is a maximal abelian subspace of $\p$, then the adjoint action of $\a$ on $\g$ is diagonalizable and the set $\Phi\subset\a^\ast$ of non-zero eigenvalues are the roots of $(\g,\a)$. It is an abstract root system in $\a^\ast$ whose Weyl group $W$ is finite. Denote by $\mf{u}_+$ the subalgebra of $\g$ generated by all positive root spaces, then $\g=\k\oplus\a\oplus\mf{u}_+$ as vector spaces. This is known as the \textit{Iwasawa decomposition} of $\g$. It is a fundamental result in the theory of Lie groups that the Cartan involution $\theta$ lifts to an involutive automorphism $\Theta:G\to G$ with differential $\theta$. Its fixed-point set $K$ is a maximal compact subgroup of $G$ (if the center of $G$ is finite) with Lie algebra $\k$ which contains all topological information about $G$ in the sense that the quotient space $G/K$ is contractible. In fact, $G/K$ is a simply connected Riemannian symmetric space of non-compact type. The exponential map $\exp:\g\to G$ is injective on $\a$ and $\mf{u}_+$ and their exponential images $A$ and $U_+$ are contractible subgroups of $G$. Here, $A$ is abelian and $U_+$ is unipotent in the sense that its Lie algebra $\mf{u}_+$ is nilpotent. The multiplication map $K\times A\times U_+\to G$ is a diffeomorphism which is called the \textit{Iwasawa decomposition} of $G$. \\
However, this map is not a homomorphism of groups which raises the question whether it is possible to describe the multiplication in $G$ in Iwasawa coordinates. More precisely, given two elements $g_1=k_1a_1u_1$ and $g_2=k_2a_2u_2$ of $G$, is it possible to compute the Iwasawa coordinates of the product $g_1g_2$ from the ones of $g_1$ and $g_2$? The major difficulty in this problem is the fact that $K$ does not normalize $A$ and $U_+$. A simpler question is to focus on one of the components and restrict $g_1$ and $g_2$ to certain subgroups of $G$. In applications, one is often interested in the $A$-component and we denote by $\Pi:G\to A,\,g=kau\mapsto a$ the corresponding projection. It is clear that $\Pi(kg)=\Pi(g)$ holds for all $k\in K$ and $g\in G$, but there is no obvious way how to compute $\Pi(gk)$. Using the Cartan decomposition $G=KAK$, this problem can be reduced to the case $g=a\in A$. \textit{Kostant's convexity theorem} asserts that if $a\in A$ and $k\in K$ are arbitrary, then $\Pi(ak)$ is contained in the (multiplicative) convex hull of the Weyl group orbit of $a$ and, as $k$ runs through all of $K$, every point of this convex hull is attained in this way. More precisely, \cite[Theorem 4.1]{Kostant} states that
	\begin{equation}
		\{\Pi(ak) \mid k\in K\}=\exp(\conv(W\cdot \log(a))). \label{ClassicalKostant}
	\end{equation}
\textit{Kac-Moody groups} are a natural generalization of semisimple Lie groups and inherit many of their properties, often by construction. It has thus been conjectured that a suitable analogue of Kostant's theorem might be true in the Kac-Moody setting as well. This question is also related to the problem of determining whether the causal pre-order on Kac-Moody symmetric spaces is a partial order or trivial (cf. \cite[Proposition 7.46]{FHHK}). In general, however, Kac-Moody groups carry no differentiable structure and an exponential map between a Kac-Moody algebra and group cannot be globally defined, so it is far more difficult to ``linearize'' problems about Kac-Moody groups. 
\medskip \newline
In this article we prove an analogue of Kostant's convexity theorem for certain classes of split real and complex Kac-Moody groups. The result is formulated as follows. Let $\A$ be a symmetrizable generalized Cartan matrix with Weyl group $W$ and let $\mc{D}$ be a free and cofree Kac-Moody root datum associated to $\A$. If $\K\in\{\R,\C\}$, the split minimal Kac-Moody group $G$ over $\K$ of type $\mc{D}$ has an Iwasawa decomposition $G=KAU_+$ generalizing the classical decomposition of a semisimple Lie group. Here, $K$ is the fixed-point set of an involutive automorphism of $G$, $A$ is the exponential image of the Cartan subalgebra $\a$ of the real Kac-Moody algebra of type $\mc{D}$ and $U_+$ is the subgroup of $G$ generated by all positive root groups. The Weyl group $W$ acts on $\a$ and $A$ and the exponential map $\exp:\a\to A$ is a $W$-equivariant group isomorphism. The subset of $\a$ where all simple roots are non-negative is called the fundamental chamber and the union $X$ of all its $W$-translates, called Weyl chambers, is the Tits cone of $\a$. In Theorem \ref{thm:KMKostant} we prove the following result.

\begin{mthm}
	Let $G$ be a free and cofree, split, minimal, real or complex Kac-Moody group, $G=KAU_+$ an Iwasawa decomposition of $G$ and $\Pi:G\to A$ the associated projection. Let $W$ denote the Weyl group of $\A$. If $a\in A$ is contained in the exponential image of an open chamber of the Tits cone $X\subset\a$, then
		$$\{\Pi(ak)\mid k\in K\}=\exp(\conv(W\cdot \log(a))).$$
\end{mthm}

\noindent If $\A$ is the Cartan matrix of a finite-dimensional complex semisimple Lie algebra, then $G$ is a reductive Lie group and the result is very similar to Kostant's original theorem. In this case, the Weyl group $W$ is finite and $X=\a$. In the general case, however, $W$ is typically infinite and $X$ is a proper subset of $\a$. Thus, it is a priori not clear whether the set $\conv(W\cdot \log(a))\subset\a$ is well-behaved in any sense. A major part of the proof of Theorem \ref{thm:KMKostant} consists of understanding the structure of this set. Our approach is inspired by the methods of \cite{Kostant}. To obtain equality between the sets in \eqref{ClassicalKostant}, Kostant uses tools from the representation theory of $G$ and an inductive argument. We proceed similarly, but our proof avoids tools that are particular to Lie theory and instead lays more emphasis on combinatorial properties of the Weyl group $W$, which is a (generally infinite) Coxeter group with respect to a finite generating set $S$. In particular, we relate the facial structure of $\conv(W\cdot \log(a))$ to the Coxeter complex of $(W,S)$ and present an inductive argument that allows to compute the Iwasawa coordinates of an element of the form $ak_1k_2$ with $k_1,k_2\in K$ by iteration. In the Lie group case, our methods can even be extended in a suitable way to provide a new proof of Kostant's original result. However, Kac-Moody groups usually do not admit a Cartan decomposition $G=KAK$, so our result does not yield a description of $\{\Pi(gk) \mid k\in K\}$ for an arbitrary $g\in G$. Another consequence of Kostant's result is that a semisimple Lie group satisfies $G=KU_+K$. As an application of Theorem \ref{thm:KMKostant}, we establish in Corollary \ref{cor:HoroDecomp} that this decomposition is not satisfied if $\A$ is not of finite type. \\ 
Furthermore, if $\A$ is of affine type, then $G$ is an example of a group with an affine $BN$-pair. Such groups are always endowed with a different variant of the Iwasawa decomposition and act on an affine building which can be viewed as a discrete analogue of a symmetric space. Kostant's convexity theorem has been transferred to the setting of affine buildings by P. Schwer\footnote{née Hitzelberger} in \cite{Schwer1} and generalized in \cite{Schwer2}. Hence, real and complex Kac-Moody groups of affine type can be endowed with two different Iwasawa decompositions that both satisfy an analogue of Kostant's convexity theorem. The result of Schwer extends far beyond this setting as it is formulated entirely in the language of buildings and does not require any group action to begin with. In addition, it also applies to groups with an affine $BN$-pair outside the Kac-Moody realm such as semisimple algebraic groups over non-Archimedean local fields.
\medskip \newline
The text is organized as follows. We begin in Chapter \ref{KMAlgebras} by reviewing some basic concepts from the theory of Kac-Moody algebras and Coxeter groups. We fix the notation that will be used throughout our discussion and introduce some standing assumptions that will be necessary in the proof of Theorem \ref{thm:KMKostant}. In Chapter \ref{ConvexHulls} we study the geometry of convex hulls of Weyl group orbits in $\a$ and obtain results that lie at the heart of the proof of our main theorem. For this, we require some notions from convex geometry that we collect in Section \ref{ConvexGeom}. Chapter \ref{KMGroups} explains the fundamental properties of Kac-Moody groups that are used in Section \ref{MainTheoremProof} where we prove Theorem \ref{thm:KMKostant}. Afterwards, we formulate a ``linear'' analogue of our main result in Section \ref{LinearAnalogue} which is closely related to a convexity theorem for Kac-Moody algebras that was established by V.G. Kac and D.H. Peterson in \cite{KP2}. We close our discussion in Section \ref{ConcludingRemarks} by comparing Theorem \ref{thm:KMKostant} to Kostant's original result and mentioning some possible generalizations and open problems.
\medskip \newline
The present article is part of the first author's PhD-project which is supported financially by the \textit{Studienstiftung des deutschen Volkes}. He would also like to thank the research group \textit{Algebra} of CAU Kiel for many valuable contributions both inside and outside of mathematics. Both authors thank Timothée Marquis for various remarks on a preliminary version of this article.

%% file: Sections/KMKostantPaper_KMAlgebras.tex
\section{Kac-Moody algebras} \label{KMAlgebras}

We begin by collecting the most important properties of Kac-Moody algebras and fix the setting in which we are going to work in the following chapters. We assume that the reader is familiar with the basic theory of these algebras and standard terminology in the subject. More details, definitions and proofs of all assertions in this chapter can be found in \cite[Chapter 1-5, 9, 11]{Kac} and \cite[Chapter 3-7]{Marquis}. Our notation and conventions mostly follow \cite{Marquis}. We also freely use standard terminology from the theory of Coxeter groups (cf. \cite[Chapter 2-3]{AB} or \cite[Chapter 1-2]{BB}). In addition, we employ the convention that we do not consider $0$ to be a natural number, i.e. $\N=\{1,2,3,\ldots\}$, and write $\N_0$ for the set $\{0,1,2,3,\ldots\}$ if that should be necessary. If $n\in\N$ is a natural number that is fixed by the context, we set $I:=\{1,2,\ldots,n\}$.
\medskip \newline
A \textit{generalized Cartan matrix} (``GCM'') is an integer matrix $\A=(a_{ij})_{i,j\in I}\in \Z^{n\times n}$ which satisfies the following three conditions for all $i,j\in I$: (i) $a_{ii}=2$, (ii) $a_{ij}\leq 0$ for $i\neq j$, (iii) $a_{ij}=0$ if and only if $a_{ji}=0$. It is said to be \textit{symmetrizable} if there exist a diagonal matrix $D$ and a symmetric matrix $B$ of size $n\times n$ such that $\A=DB$. It is called \textit{decomposable} if $I$ is the disjoint union $I=I_1\sqcup I_2$ of two non-empty proper subsets $I_1,I_2\subset I$ such that $a_{ij}=0$ for all $i\in I_1$ and $j\in I_2$. In this case, the submatrices $\A_1$ and $\A_2$ of $\A$ indexed by $I_1$ and $I_2$ are GCM's as well and we write $\A=\A_1\oplus\A_2$. Otherwise, $\A$ is said to be \textit{indecomposable}. There is a complete classification of indecomposable GCM's that is formulated as follows (cf. \cite[Theorem 4.3]{Kac}): For a vector $\lambda=(\lambda_1,\ldots,\lambda_n)^T\in\R^n$, we write $\lambda>0$ if $\lambda_i>0$ for all $i\in I$ and, likewise, we define $\lambda<0$, $\lambda\geq 0$ and $\lambda\leq 0$. If $\A$ is indecomposable, then exactly one of the following conditions holds. According to these three cases, we say that $\A$ is of \textit{finite}, \textit{affine} or \textit{indefinite type}. 
	\begin{itemize}
		\item[(Fin)] $\det(\A)\neq 0$; there exists $\lambda>0$ such that $\A\lambda>0$; $\A\lambda\geq 0$ implies $\lambda>0$ or $\lambda=0$ for all $\lambda\in\R^n$.
		\item[(Aff)] $\op{rank}(\A)=n-1$, there exists $\lambda>0$ such that $\A\lambda=0$; $\A\lambda\geq 0$ implies $\A\lambda=0$ for all $\lambda\in\R^n$.
		\item[(Ind)] There exists $\lambda>0$ such that $\A\lambda<0$; $\A\lambda\geq 0$ and $\lambda\geq 0$ imply $\lambda=0$ for all $\lambda\in\R^n$.
	\end{itemize}
A \textit{Kac-Moody root datum} is a quintuple $\mc{D}=(I,\A,\Lambda,(c_i)_{i\in I},(h_i)_{i\in I})$, where $I$ is a set indexing a GCM $\A$ of size $n=|I|$, $\Lambda$ is a free $\Z$-module of rank $d\in\N$, $(c_i)_{i\in I}$ is a family of elements of $\Lambda$ and $(h_i)_{i\in I}$ is a family in the dual $\Z$-module $\Lambda^\vee$ such that the natural pairing between $\Lambda$ and $\Lambda^\vee$ satisfies $\langle c_j,h_i \rangle=a_{ij}$ for all $i,j\in I$. We say that $\mc{D}$ is \textit{free} (resp. \textit{cofree}) if $(c_i)_{i\in I}\subset\Lambda$ (resp. $(h_i)_{i\in I}\subset\Lambda^\vee$) is linearly independent. Let $\K$ be a field. The \textit{Kac-Moody algebra of type $\mc{D}$ over $\K$} is the free Lie algebra $\g$ over $\K$ generated by $\h:=\Lambda^\vee\otimes_\Z \K$ and $2n$ symbols $\{e_i,f_i \mid i\in I\}$ subject to the following relations:
	\begin{equation}
	\begin{aligned} \label{DefAlgRelations}
		[h,h'] &= 0 \text{ for $h,h'\in\h$} \\
		[h,e_i] &= \langle c_i,h\rangle e_i \text{ and } [h,f_i]=-\langle c_i,h\rangle f_i \text{ for $h\in\h$ and $i\in I$} \\
		[e_i,f_j] &= -\delta_{ij}h_i \text{ for $i,j\in I$} \\
		\ad(e_i)^{1-a_{ij}}e_j &= \ad(f_i)^{1-a_{ij}}f_j=0 \text{ for $i,j\in I$ with $i\neq j$}
	\end{aligned}
\end{equation}
We call $\h$ the \textit{Cartan subalgebra} and $\{e_i,f_i \mid i\in I\}$ the \textit{Chevalley generators} of $\g$. The \textit{Chevalley involution} of $\g$ is the involutive $\K$-linear Lie algebra automorphism $\omega:\g\to\g$ defined by $\omega(e_i):=f_i$ and $\omega(f_i):=e_i$ for all $i\in I$ and $\omega(h):=-h$ for all $h\in\h$. Let $\Phi_0:=\{\alpha_i \mid i\in I\}$ and $\Phi_0^\vee:=\{\alpha_i^\vee \mid i\in I\}$ be two abstract sets and denote by
	$$Q:=\bigoplus_{i\in I} \Z\alpha_i \qquad \text{and} \qquad Q^\vee:=\bigoplus_{i\in I} \Z\alpha_i^\vee$$
the free abelian groups generated by them. Then $\g$ has a natural $Q$-gradation
	$$\g=\bigoplus_{\alpha\in Q} \g_\alpha$$
induced by the assignment $\op{deg}(e_i)=-\op{deg}(f_i):=\alpha_i$ for $i\in I$ and $\op{deg}(h):=0$ for $h\in\h$. The elements of $\Phi:=\{\alpha\in Q\setminus\{0\}:\g_\alpha\neq\{0\}\}$ are the \textit{roots} of $\g$, $\Phi_0\subset\Phi$ is the set of \textit{simple roots} and $\Phi_0^\vee$ is the set of \textit{simple coroots}. Every root is a $\Z$-linear combination of simple roots with coefficients of the same sign, which defines two sets $\Phi_+,\Phi_-\subset\Phi$, where
	$$\Phi_{\pm}:=\Phi\cap Q_{\pm} \qquad \text{and} \qquad Q_{\pm}:=\pm\bigoplus_{i\in I} \N_0\alpha_i,$$
that satisfy $\Phi_+=-\Phi_-$ and $\Phi=\Phi_+\sqcup\Phi_-$. We introduce a partial order on $Q$ by declaring $\alpha\leq\beta$ if and only if $\beta-\alpha\in Q_+$. For $\alpha=\sum_{i\in I} k_i\alpha_i\in Q$ with $k_1,\ldots,k_n\in\Z$, the \textit{support} and \textit{height} of $\alpha$ are defined by
	$$\op{supp}(\alpha):=\{i\in I: k_i \neq 0\} \qquad \text{and} \qquad \op{ht}(\alpha):=\sum_{i\in I} k_i$$
and one defines the same notions for elements of $Q^\vee$ in the obvious way. There exist canonical $\Z$-linear maps $c:Q\to\Lambda$ and $h:Q^\vee\to\Lambda^\vee$ that satisfy $c(\alpha_i)=c_i$ and $h(\alpha_i^\vee)=h_i$ for all $i\in I$. If $\mc{D}$ is free (resp. cofree), then $c$ (resp. $h$) is injective. The Cartan subalgebra $\h=\g_0\subset\g$ acts diagonally on $\g$ via the adjoint action and we have $\g_{\alpha}\subset \{x\in\g: [h,x]=\langle c(\alpha),h\rangle x\,\,\forall h\in\h\}$ for all $\alpha\in\Phi$, but this inclusion is strict if $\mc{D}$ is not free. Hence, the eigenspace decomposition for the adjoint action of $\h$ is usually coarser than the $Q$-gradation of $\g$.

\begin{rmk} \label{rmk.RDissues}
	It is immediate from the property $\langle c_j,h_i\rangle=a_{ij}$ for all $i,j\in I$ in the definition of a Kac-Moody root datum that if $\mc{D}$ is not cofree (resp. free), then the rows (resp. columns) of the GCM $\A$ are linearly dependent over $\Z$. Conversely, every $\Z$-linear dependence between the rows (resp. columns) of $\A$ can be used to construct a root datum of rank $n$ that is not cofree (resp. free). Hence, a GCM is invertible if and only if all associated root data are free and cofree. In the non-invertible case, there always exist Kac-Moody root data of rank $d<n$ which are therefore neither free nor cofree. \\
	If $\A$ is arbitrary, there exists a Kac-Moody root datum $\mc{D}^{\A}_{\text{Kac}}$ that is free and cofree such that $d=2n-\op{rank}(\A)$ (cf. \cite[Example 7.10]{Marquis}). It is minimal with these properties in the sense that every free and cofree root datum has to satisfy $d\geq 2n-\op{rank}(\A)$. The Kac-Moody algebra of type $\mc{D}^{\A}_{\text{Kac}}$ over $\C$ coincides with the algebra $\g(\A)$ defined in \cite[Definition 3.17]{Marquis}. Similarly, there exists a cofree Kac-Moody root datum $\mc{D}^{\A}_{\text{sc}}$ of rank $d=n$ such that the associated complex Kac-Moody algebra is the derived algebra of $\g(\A)$ (cf. \cite[Example 7.11]{Marquis}). If $\A$ is invertible, these root data are the same. If $\A$ is symmetrizable, then by the \textit{Gabber-Kac theorem} (cf. \cite[Theorem 9.11]{Kac}), these two definitions also coincide (up to some minor conventions; see \cite[Remark 3.19]{Marquis}) with the ``classical'' Kac-Moody algebra (resp. its derived algebra) that was first introduced by V.G. Kac and studied in \cite{Kac}.
\end{rmk}

\noindent The \textit{Weyl group} associated to $\A$ is the group $W$ generated by the set $S:=\{s_i \mid i\in I\}$ subject to the relations $(s_is_j)^{m_{ij}}=1$ for all $i,j\in I$, where $m_{ij}$ is equal to $1$ if $i=j$ and determined by the following table if $i\neq j$:
	\begin{equation} \label{CoxeterTable}
		\begin{array}{c|c|c|c|c|c}
			a_{ij}a_{ji} & 0 & 1 & 2 & 3 & \geq 4 \\
			\hline 
			m_{ij} & 2 & 3 & 4 & 6 & \infty
		\end{array}
	\end{equation}
By definition, $(W,S)$ is a Coxeter system with Coxeter matrix $M=(m_{ij})_{i,j\in I}$ and we denote by $l:W\to\N_0$ the length function of $W$ with respect to the generating set $S$. The Weyl group operates on $Q$ and $\Lambda$, where the action of the generator $s_i$ $(i\in I)$ is given by the following maps:
	\begin{align*}
		Q &\to Q &\qquad \Lambda &\to \Lambda \\
		\beta &\mapsto \beta-\langle\beta,\alpha_i^\vee\rangle\alpha_i &\qquad \lambda&\mapsto \lambda-\langle \lambda,h_i\rangle c_i
	\end{align*}
By duality, $W$ also acts on $Q^\vee$ and $\Lambda^\vee$ such that the natural pairings between $Q$ and $Q^\vee$ (resp. $\Lambda$ and $\Lambda^\vee)$ are $W$-invariant. In particular, these actions induce $\K$-linear representations of $W$ on $\h$ and $\h^\ast$ by extension of scalars. By definition, the canonical maps $c:Q\to\Lambda$ and $h:Q^\vee\to\Lambda^\vee$ are $W$-equivariant. The $W$-actions on $Q$ and $Q^\vee$ are faithful while the actions on $\Lambda$ and $\Lambda^\vee$ (and hence on $\h$ on $\h^\ast$) are not necessarily faithful if $\mc{D}$ is neither free nor cofree. Since the actions of $W$ on $\Lambda$ and $\Lambda^\vee$ are dual, it suffices that one of the maps $c$ or $h$ is injective, i.e. that $\mc{D}$ is either free or cofree, in order to obtain a faithful action. The set of roots $\Phi\subset Q$ is $W$-invariant and the subset $\Phi^{\op{re}}:=W\cdot\Phi_0\subset\Phi$ is called the set of \textit{real roots}. Its complement $\Phi^{\op{im}}:=\Phi\setminus\Phi^{\op{re}}$, called the set of \textit{imaginary roots}, will not play a major role in our discussion. If $\alpha\in\Phi^{\op{re}}$, i.e. $\alpha=w(\alpha_i)$ for some $w\in W$ and $i\in I$, the element $\alpha^\vee:=w(\alpha_i^\vee)\in Q^\vee$ is independent of the representation of $\alpha$ as a $W$-translate of a simple root and the associated reflection 
	$$s_{\alpha}:Q\to Q, \qquad \beta\mapsto \beta-\langle\beta,\alpha^\vee\rangle\alpha$$
satisfies $s_\alpha=ws_iw^{-1}\in W$. For every $w\in W$ and $i\in I$, we have
\begin{equation}
	w(\alpha_i)\in\Phi_+ \Longleftrightarrow l(w)<l(ws_i). \label{PositivityCriterion}
\end{equation}
We use the following standard terminology from the theory of Coxeter groups. Given an element $w\in W$, there is a unique set $S(w)\subset S$ such that every reduced decomposition of $w$ contains precisely the reflections in $S(w)$. We call $\op{supp}(w):=\{i\in I:s_i\in S(w)\}$ the \textit{support} of $w$. For any subset $J\subset I$, the subgroup $W_J$ of $W$ generated by $S_J:=\{s_j \mid j\in J\}$ is called a \textit{standard subgroup} of $W$. A left coset $wW_J$ with $w\in W$ is called a \textit{standard coset}. If $J=I\setminus\{i\}$ for some $i\in I$, we write $S^{(i)}$ and $W^{(i)}$ instead of $S_{I\setminus\{i\}}$ and $W_{I\setminus\{i\}}$. The \textit{Coxeter complex} of $(W,S)$ is the partially-ordered set (``poset'') of all standard cosets of $(W,S)$ ordered by reversed inclusion and denoted by
$$\Sigma(W,S)=\{wW_J \mid w\in W,\, J\subset I\}.$$
We also write $\Sigma^\ast(W,S)$ for the poset with the same underlying set but ordered by non-reversed inclusion. Every standard coset $wW_J\subset W$ has a unique element $w'$ of minimal length which is characterized by $l(w'w_J)=l(w')+l(w_J)$ for all $w_J\in W_J$.
\medskip \newline
Next, we introduce some notions from the representation theory of $\g$ (cf. \cite[Section 4.1]{Marquis} or \cite[Chapter 3 and 9]{Kac}). A representation $\rho$ of $\g$ on a $\K$-vector space $V$ is said to be \textit{integrable} if $\h$ acts diagonally on $V$ and for every $i\in I$ and $v\in V$ there exists some $N\in\N$ such that $\rho(e_i)^Nv=\rho(f_i)^Nv=0$. A functional $\lambda\in\h^\ast$ is called a \textit{dominant integral weight} if it satisfies $\langle\lambda,h_i\rangle\in \N_0$ for all $i\in I$. For every such weight, there exists a unique irreducible integrable $\K$-linear representation $\rho_\lambda:\g\to\mf{gl}(L(\lambda))$ with highest weight $\lambda$. This means that there exists a non-zero vector $v_\lambda\in L(\lambda)$, which is unique up to scalars, such that $\rho_\lambda(h)v_\lambda=\langle \lambda,h\rangle v_\lambda$ for all $h\in\h$, the subalgebra $\mf{u}_+:=\bigoplus_{\alpha\in\Phi_+} \g_\alpha\subset\g$ annihilates $v_\lambda$ and $L(\lambda)$ is generated by $v_\lambda$ as a representation of the universal enveloping algebra of $\g$. The Cartan subalgebra $\h\subset\g$ acts diagonally on $L(\lambda)$, its joint eigenvalues $\op{wt}(L(\lambda))\subset\h^\ast$ are the \textit{weights} of $L(\lambda)$ and the joint eigenspace decomposition
	$$L(\lambda)=\bigoplus_{\mu\in\op{wt}(L(\lambda))} L(\lambda)_\mu$$
is the \textit{weight space decomposition} of $L(\lambda)$. Every weight is of the form $\lambda-c(\alpha)$ for some $\alpha\in Q_+$ and the set of weights is $W$-invariant. If $\mu$ is any weight and $x\in \g_\alpha$ for some $\alpha\in\Phi$, then it is immediate from the defining relations \eqref{DefAlgRelations} of $\g$ that 
\begin{equation}
	\rho_\lambda(x)(L(\lambda)_\mu)\subset L(\lambda)_{\mu+c(\alpha)}. \label{WeightSpaceTranslation}
\end{equation} 
Depending on the Kac-Moody root datum $\mc{D}$, there does not necessarily exist a non-zero dominant integral weight. If $\mc{D}$ is cofree, however, then $\{h_i\mid i\in I\}$ can be extended to a $\K$-basis of $\h$ and the corresponding dual basis consists of dominant integral weights. If $\langle\lambda,h_i\rangle=0$ for all $i\in I$, then the defining relations of $\g$ imply that we may define a representation $\rho_\lambda$ of $\g$ on the one-dimensional vector space $\K$ by setting $\rho_\lambda(e_i)=\rho_\lambda(f_i):=0$ for all $i\in I$ and $\rho_\lambda(h)1:=\langle \lambda,h\rangle$ for all $h\in\h$. This clearly defines an irreducible representation of $\g$ with highest weight $\lambda$. Since $L(\lambda)$ is unique up to isomorphism, it follows that $L(\lambda)$ is one-dimensional and isomorphic to $\K$ under the map $v_\lambda\mapsto 1$.

\begin{assum} \label{assum:GlobalAssumption}
	From now on, we will always make the following assumptions which will play a central role in the proof of our main theorem:
	\begin{itemize}
		\item[(i)] The GCM $\A$ is symmetrizable.
		\item[(ii)] The Kac-Moody root datum $\mc{D}$ is free and cofree.
		\item[(iii)] The field $\K$ is either $\R$ or $\C$.
	\end{itemize}
\end{assum}

\noindent The first assumption occurs frequently in the theory of Kac-Moody algebras and will be used explicitly in the proof of Theorem \ref{thm:KMKostant}. The necessity of the third will already be apparent in the formulation of that result as it involves an ``exponential function'' on the Cartan subalgebra of $\g$ whose definition is based on the usual real or complex exponential map. In our discussion, it will usually be possible to treat real and complex Kac-Moody algebras at the same time, so we do not want to distinguish them in notation. Hence, for the rest of this article, we denote by $\g$ either the real or the complex Kac-Moody algebra of type $\mc{D}$. If we need to emphasize the field, we write $\g_\R$ or $\g_\C$. It will, however, be very important to distinguish the real Cartan subalgebra from the complex one, so we will write $\a:=\Lambda^\vee\otimes_\Z \R$ and $\h:=\Lambda^\vee\otimes_\Z \C$. We explicitly do not denote these algebras by $\h_\R$ and $\h_\C$ since we often want to view $\a$ as a subset of $\g_\C$ and a notation like $\h_\R\subset\g_\C$ might be misleading. We denote by $\theta:\g\to\g$ the involutive automorphism of $\g$ obtained from pre-composing the Chevalley involution of $\g$ with complex conjugation. More precisely, if $\K=\R$, then $\theta$ coincides with the Chevalley involution $\omega_\R$ of $\g_\R$, and if $\K=\C$, then $\theta$ is the unique extension of $\omega_\R$ to a $\C$-antilinear automorphism of $\g_\C$. In both cases, we call $\theta$ the \textit{compact involution} of $\g$ and denote its fixed-point set by $\k$.
\medskip \newline
Assumption (ii) is a simplification that requires some justification. It allows us to identify $Q$ and $Q^\vee$ with subsets of $\Lambda$ and $\Lambda^\vee$. For every $i\in I$, we will therefore write $\alpha_i$ for $c_i$ and $\alpha_i^\vee$ for $h_i$ and suppress the inclusions $c$ and $h$ from the notation. With this convention, we may extend $\Phi_0^\vee\subset\Lambda^\vee$ to an $\R$-basis $\mc{B}:=\{\alpha_1^\vee,\ldots,\alpha_n^\vee,h_{n+1},\ldots,h_d\}$ of $\a$, which also yields a $\C$-basis of $\h$, and we denote by $\mc{B}^\ast:=\{\omega_1,\ldots,\omega_n,\phi_{n+1},\ldots,\phi_d\}\subset\a^\ast$ the corresponding dual basis. 
Recall that the assumption that $\mc{D}$ is free and cofree implies $d\geq 2n-\op{rank}(\A)\geq n$. The complex Kac-Moody algebra $\g_\C$ is then simply a direct product of the Kac-Moody algebra $\g(\A)$ studied in \cite{Marquis}, which coincides with the definition in \cite{Kac} by assumption (i) and the Gabber-Kac theorem. More precisely, we have $\g_\C\cong\g(\A)\oplus \h'$ as a Lie algebra, where $\h'$ is a subspace of $\h$ of dimension $d-2n+\op{rank}(\A)$ on which all roots vanish. The obvious analogue of this assertion holds in the real case. This allows us to use most results of \cite{Kac} and \cite{Marquis} in our setting. Simply put, we study Kac-Moody algebras that behave exactly like the classical ones, except that we allow the Cartan subalgebra to have dimension larger than $2n-\op{rank}(\A)$. As discussed in Remark \ref{rmk.RDissues}, condition (ii) is always satisfied if $\A$ is invertible.
\medskip \newline
Moreover, the assumption that $\mc{D}$ is free guarantees that the $\R$-linear representation $W\to\GL(\a)$ is faithful and the set
$$C:=\{h\in\a: \langle\alpha_i,h\rangle>0 \,\, \forall i\in I\}$$
is not empty. It is called the \textit{open fundamental chamber} and its closure (w.r.t. the standard Euclidean topology on $\a$)
$$\overline{C}:=\{h\in\a: \langle\alpha_i,h\rangle\geq 0 \,\, \forall i\in I\}$$
is the \textit{closed fundamental chamber}. It can equivalently be described as
\begin{equation}
	\overline{C}=\bigg\{h\in\a: \forall w\in W: h-w(h)=\sum_{i=1}^n c_i\alpha_i^\vee \text{ for some $c_1,\ldots,c_n\in \R_{\geq 0}$}\bigg\}. \label{fundamentalChamber}
\end{equation}
The $W$-translates of $\overline{C}$ are called \textit{Weyl chambers}. Every such set $w(\overline{C})\subset\a$ with $w\in W$ is a simplicial cone which is bounded by the \textit{walls} $\{w(V_{\alpha_i}) \mid i\in I\}$, where $V_{\alpha_i}:=\op{ker}(\alpha_i)$. The union of all Weyl chambers
$$X:=\bigcup_{w\in W} w(\overline{C})\subset\a$$
is the \textit{Tits cone} of $\a$. The closed fundamental chamber is a strict fundamental domain for the $W$-action on $X\subset\a$, i.e. every point of $X$ lies in the $W$-orbit of a unique point in $\overline{C}$. If $h\in\overline{C}$ and $J:=\{i\in I:\langle\alpha_i,h\rangle=0\}$, then the stabilizer of $h$ is the standard subgroup $W_J\subset W$ generated by the reflections $S_J=\{s_j \mid j\in J\}$.
\medskip \newline
Finally, Assumption \ref{assum:GlobalAssumption} also has one important consequence for the representation theory of $\g$. If $\lambda\in\a^\ast\subset\h^\ast$ is a dominant integral weight, then by \cite[Proposition 9.4 and Lemma 11.5]{Kac} there exists a non-degenerate Hermitian form $H:L(\lambda)_\C\times L(\lambda)_\C\to\C$ which is \textit{contravariant} in the sense that
	\begin{equation}
		H(x\cdot v,w)=-H(v,\theta(x)\cdot w) \qquad \text{for all $x\in\g_\C$ and $v,w\in L(\lambda)_\C$.} \label{ContravariantForm}
	\end{equation}
It is unique up to scalars and different weight spaces are $H$-orthogonal. Since $\lambda|_\a$ is real valued, it restricts to a non-degenerate symmetric bilinear form on $L(\lambda)_\R$. If the GCM $\A$ is symmetrizable, then $H$ is positive definite by \cite[Theorem 11.7]{Kac}. In particular, if $\mc{D}$ is free and cofree, the dual basis $\mc{B}^\ast\subset\a^\ast$ consists of dominant integral weights. For ease of reference, we summarize the most important consequences of Assumption \ref{assum:GlobalAssumption} in the following proposition.

\begin{prop} \label{prop:AssumptionConsequences}
	Suppose that Assumption \ref{assum:GlobalAssumption} is in effect, then $d\geq 2n-\op{rank}(\A)$ and the following three properties hold:
		\begin{itemize}
			\item[(i)] The dual basis $\mc{B}^\ast=\{\omega_1,\ldots,\omega_n,\phi_{n+1},\ldots,\phi_d\}\subset\a^\ast$ consists of dominant integral weights. Each of the real or complex highest-weight modules $L(\omega_1),\ldots,L(\omega_n),L(\phi_{n+1}),\ldots,L(\phi_d)$ can be endowed with a positive definite, contravariant, symmetric bilinear or Hermitian form $H$ such that weight spaces belonging to different weights are $H$-orthogonal. 
			\item[(ii)] For every $i\in\{n+1,\ldots,d\}$, the highest-weight module $L(\phi_i)$ is one-dimensional. The subalgebras $\mf{u}_+$ and $\mf{u}_-$ of $\g$ act trivially on $L(\phi_i)$ and an element $h\in\h$ acts on $L(\phi_i)$ by multiplication with $\langle \phi_i,h\rangle$.
			\item[(iii)] The open fundamental chamber $C\subset\a$ is not empty.
		\end{itemize}
\end{prop}

%% file: Sections/KMKostantPaper_ConvexHulls.tex
\section{Convex hulls of Weyl group orbits} \label{ConvexHulls}

The goal of this chapter is to study the geometry of $W$-orbits in $\a$ and relate it to the combinatorial structure of the Coxeter system $(W,S)$. More precisely, we will be interested in the convex hull of the $W$-orbit of an element $h\in\a$:
	$$\Xi_h:=\conv(W\cdot h)$$
The structure of this set depends heavily on the location of $h$ with respect to the double Tits cone $X\cup -X\subset\a$. One can distinguish the following three possibilities. We say that $h$ is
\begin{itemize}
	\item \textit{regular} if $h$ or $-h$ is contained in an open chamber of $X$.
	\item \textit{singular} if $h$ or $-h$ is contained in a wall of $X$.
	\item \textit{external} if neither $h$ nor $-h$ is contained in $X$.
\end{itemize}
In the present article, we will restrict our attention to regular elements, which exist since $\mc{D}$ is assumed to be free. In this case, we shall see that $\Xi_h$ has a simple combinatorial description in terms of the Coxeter complex $\Sigma(W,S)$. Some of our results do indeed generalize to the singular case and we briefly discuss this in Section \ref{ConcludingRemarks}. We emphasize that Assumption \ref{assum:GlobalAssumption} and the notational conventions following it are in effect for this entire chapter. 

\begin{rmk}
	If the Weyl group $W$ is finite, then $\Xi_h$ is a convex polytope (see Section \ref{ConvexGeom} for the definition) and most results of Section \ref{regularOrbit} can be deduced from the general theory of these sets. However, this is the case if and only if $\A$ is the Cartan matrix of a finite-dimensional semisimple Lie algebra which is a very small subclass of Kac-Moody algebras. If $W$ is infinite, it is not even a priori clear whether $\Xi_h$ is a closed subset of $\a$. 
\end{rmk}

\noindent \textbf{Convention:} From this point onward, the set of imaginary roots $\Phi^{\op{im}}$ will not appear in our discussion. Thus, we simplify the notation and denote the set of real roots by $\Phi$ instead of $\Phi^{\op{re}}$.

\subsection{Some results from convex geometry} \label{ConvexGeom}

To prepare for the study of $\Xi_h$, we begin by collecting some basic properties of convex sets in finite-dimensional real vector spaces. We follow the terminology in \cite{Rockafellar}. More details and proofs of the assertions about general convex sets can be found in \cite[§1, 6, 11, 17, 18]{Rockafellar} or \cite[Chapter 3-5]{Gruber}. A thorough discussion of convex polytopes that goes way beyond what is needed for this article is provided by \cite[Chapter 14-20]{Gruber}.
\medskip \newline
Let $V$ be a real vector space of dimension $d<\infty$ equipped with its standard Euclidean topology. For two points $x,y\in V$, we denote by $[x,y]$ the closed line segment connecting $x$ to $y$ and by $(x,y)$, $[x,y)$ and $(x,y]$ the corresponding open and half-open line segments. Let $P\subset V$ be a convex subset, not necessarily closed. A \textit{face} of $P$ is a non-empty convex subset $F\subset P$ such that whenever $p_1,p_2\in P$ satisfy $(p_1,p_2)\cap F\neq\emptyset$, then we already have $[p_1,p_2]\subset F$. The \textit{dimension} of a face is the dimension of its affine span. We write $\mc{F}(P)$ for the partially-ordered set of faces of $P$ ordered by inclusion. Observe that $P$ is trivially a face of itself, but we have excluded the empty set from being a face. This is merely a convention that will simplify the terminology in Section \ref{regularOrbit}. Hence, $\mc{F}(P)$ has a maximal element, but not necessarily a minimal one. The intersection of any two faces of $P$ is either empty or again a face of $P$ and for any face $F$ of $P$, the faces of $F$ are precisely the faces of $P$ contained in $F$. If $A\subset P$ is a convex subset and $F$ a face of $P$ which contains a point from the relative interior of $A$, then it follows immediately from the definition of a face that $A\subset F$. In particular, every proper face $F\subsetneq P$ is contained in the topological boundary of $P$ and the relative interiors of different faces are disjoint. For two points $p_1,p_2\in P$, we therefore have $[p_1,p_2]\subset \partial P$ if and only if $p_1$ and $p_2$ are contained in a common proper face. Note that non-trivial faces do not exist if $P\cap\partial P=\emptyset$. \\
We will mostly be interested in a less general notion. Suppose that $E\subset V$ is an affine hyperplane\footnote{By an \textit{affine hyperplane} in $V$, we always mean an affine subspace of codimension $1$.} such that $P$ lies in one of the two closed half-spaces defined by $E$ and $F:=P\cap E\neq\emptyset$. Then $F$ is a face of $P$ and faces of this type are said to be \textit{exposed}. The plane $E$ is called a \textit{supporting hyperplane} of $P$ and the half-space containing $P$ is a \textit{supporting half-space}. Exposed faces of dimension $0$ and $1$ are called \textit{vertices} and \textit{edges}, those of dimension $\dim(P)-1$ are called \textit{facets} and their supporting hyperplanes are referred to as \textit{defining hyperplanes}. If $\dim(P)=d$, then defining hyperplanes are uniquely determined and, more generally, two defining hyperplanes give rise to the same face if and only if their intersections with the affine span of $P$ coincide. The trivial face $P$ is clearly exposed if $\dim(P)<d$ since every affine hyperplane $E$ containing $P$ is supporting, but it is trivial in the sense that $P$ is contained in both closed half-spaces defined by $E$. We also consider $P$ to be exposed if $\dim(P)=d$. But $P$ may have non-exposed proper faces which are then necessarily contained in $\partial P$. However, it follows from a separation argument (cf. \cite[Theorem 11.6]{Rockafellar}) that every point of $P\cap \partial P$ is contained in a non-trivial supporting hyperplane of $P$. Consequently, every proper face of $P$ is contained in a proper exposed face. In particular, if $P$ is closed, then every point of $P$ is either contained in the relative interior of $P$ or lies on a non-trivial supporting hyperplane, which implies that $P$ can be reconstructed as the intersection of all supporting half-spaces containing it. In algebraic terms, this means that a point is contained in $P$ if and only if it satisfies a family of linear (in-)equalities. If $\dim(P)=d$, then every supporting hyperplane is non-trivial and none of these inequalities are equalities. By replacing $V$ with the affine span of $P$, it is usually no restriction to assume that $\dim(P)=d$, but for our purposes, it will be more convenient to avoid this assumption. \\
As mentioned above, the faces of a face $F$ of $P$ are exactly the faces of $P$ contained in $F$. In particular, this holds for exposed faces. More generally, suppose that $E\subset V$ is an arbitrary hyperplane, not necessarily supporting. If $P':=P\cap E$ is non-empty, then for every face $F$ of $P$, the intersection $F\cap E$ is either empty or a face of $P'$. Conversely, it is not difficult to show that for every face $F'$ of $P'$ there exists a face $F$ of $P$ such that $F'=F\cap E$. The intersection of all such faces of $P$ also has this property, so if $F$ is assumed to be minimal, it is uniquely determined and its dimension is given by
	\begin{equation}
		\op{dim}(F)=\begin{cases} \op{dim}(F') & \text{if $F'$ is a face of $P$ contained in $E$} \\ \op{dim}(F')+1 & \text{otherwise.}\end{cases} \label{FaceIntersection}
	\end{equation}
For a set $Y\subset V$, the \textit{convex hull} of $Y$ is the smallest convex subset of $V$ containing $Y$. It consists precisely of all finite convex combinations $\sum_{y\in Y}\lambda_y y$, where $\lambda_y\in [0,1]$ for all $y\in Y$, $\lambda_y>0$ only for finitely many $y\in Y$ and $\sum_{y\in Y}\lambda_y=1$. It is denoted by $\conv(Y)$. If $Y$ is finite, then $P:=\conv(Y)$ is compact and called a \textit{convex polytope}. A convex polytope $P$ has only finitely many faces which are all exposed. Every face $F$ of $P$ is of the form $F=\conv(Y')$ for some subset $Y'\subset Y$ and thus again a convex polytope. For every $0\leq k\leq\dim(P)$, there exists a $k$-dimensional face of $P$ and all maximal chains in $\mc{F}(P)$ have the same length. In particular, every non-trivial face of $P$ is contained in a facet, i.e. a maximal proper face. \\
If $Y$ is infinite, however, it is not necessarily true that $P:=\conv(Y)$ is closed, not even if $Y$ is uniformly discrete. It might even be open in which case it cannot have any non-trivial faces. Hence, the theory of convex polytopes cannot be applied in this situation and it is not clear whether $\mc{F}(P)$ has any interesting properties. However, it is still true that if $F\subset P$ is an exposed face, then there exists a subset $Y'\subset Y$ such that $F=\conv(Y')$. 
	
\subsection{Orbits of regular elements} \label{regularOrbit}

We now turn to the study of $\Xi_h=\conv(W\cdot h)$ where $h\in\a$ is regular. After possibly replacing $h$ by $-h$, it is no restriction to assume that $h$ is contained in an open chamber of $X$. In this case, the $W$-translates $\{w(h) \mid w\in W\}$ are all distinct. Since $\Xi_h$ is $W$-invariant, we may w.l.o.g. assume that $h\in C$, i.e. $\langle\alpha_i,h\rangle>0$ for all $i\in I$. Our goal is to show that $\Xi_h$ is a closed convex set whose face poset $\mc{F}(\Xi_h)$ is isomorphic to $\Sigma^\ast(W,S)$, i.e. the poset of all standard cosets of $(W,S)$ ordered by non-reversed inclusion. Our approach will consist of two steps: First, we explicitly construct exposed faces of $\Xi_h$. Afterwards, we show that $\Xi_h$ is closed and that every face is exposed.
\medskip \newline
First of all, note that the assumption that $\mc{D}$ is cofree implies $\dim(\Xi_h)=n$ since it is contained in the $n$-dimensional affine subspace $E_h:=h+\a'$, where $\a':=\bigoplus_{i=1}^n \R\alpha_i^\vee$, and the $n+1$ points $\{h,s_1(h),\ldots,s_n(h)\}\subset\Xi_h$ are affinely independent. Next, $W$-invariance of $\Xi_h$ also implies that the set of faces of $\Xi_h$ is $W$-invariant. Hence, every face of $\Xi_h$ is $W$-equivalent to a face that meets $\overline{C}$. The same is also true for the subset of exposed faces. Every exposed face $F$ of $\Xi_h$ is of the form $F=\op{conv}(W'\cdot h)$ where $W'$ is a subset of $W$, so by $W$-invariance, it is also no restriction to assume that $1\in W'$ and hence $h\in F$. Thus, it is sufficient to construct exposed faces containing $h$. To this end, we define for every $i\in I$ the hyperplane
	$$H^{(i)}:=\{h'\in E_h: \langle \omega_i,h' \rangle =\langle \omega_i,h \rangle\}\subset E_h$$
In order to show that these hyperplanes give rise to faces of $\Xi_h$, we need the following lemma from the general theory of Coxeter groups, which is a refinement of \eqref{fundamentalChamber}. 

\begin{lem}\label{lem:weyl_difference}
	If $h\in C$, then $h-w(h)=\sum_{i\in I} c_i\alpha_i^\vee$ with $c_i>0$ for all $i\in\op{supp}(w)$ and $c_i=0$ for all $i\notin \op{supp}(w)$.
\end{lem}

\begin{proof}
	By \eqref{fundamentalChamber} and the fact that $c_i=0$ if $i\notin\op{supp}(w)$, it only remains to show that $c_i>0$ for all $i\in\op{supp}(w)$. Suppose that this was not the case, then we would have $\langle\omega_i,w(h)\rangle=\langle\omega_i,h\rangle$ for some $i\in\op{supp}(w)$. But then every element in the left coset $W^{(i)}w$ satisfies the same equation, so we may assume that $w$ is the unique element of minimal length in that coset. In this case, it has the form $w=s_iw'$ for some $w'\in W$ with $l(w')=l(w)-1$ and we have
		$$\langle\omega_i,h\rangle=\langle\omega_i,w(h)\rangle=\langle s_i(\omega_i),w'(h)\rangle=\langle\omega_i,w'(h)\rangle-\langle\alpha_i,w'(h)\rangle$$
	with $\langle\omega_i,w'(h)\rangle\leq \langle\omega_i,h\rangle$ by \eqref{fundamentalChamber} and $\langle\alpha_i,w'(h)\rangle=\langle w'^{-1}(\alpha_i),h\rangle>0$ by \eqref{PositivityCriterion} because $l(w'^{-1})=l(w')<l(w)=l(w^{-1})=l(w'^{-1}s_i)$. Hence, we arrive at a contradiction.
\end{proof}

\begin{prop} \label{prop:maxface}
	For every $i\in I$, the plane $H^{(i)}$ is a defining hyperplane of $\Xi_h$ and the corresponding facet of $\Xi_h$ is given by 
	$$F^{(i)}:=\op{conv}(W^{(i)}\cdot h),$$
	where $W^{(i)}$ is the subgroup of $W$ generated by $S^{(i)}=S\setminus\{s_i\}$.
\end{prop}

\begin{proof}
	If $h'\in \Xi_h$, we can write it as $h'=\sum_{w\in W} \lambda_w w(h)$ with $\lambda_w\in [0,1]$ for all $w\in W$ and $\sum_{w\in W} \lambda_w=1$, where $\lambda_w>0$ only for finitely many $w\in W$. Moreover, it follows from \eqref{fundamentalChamber} that $\langle \omega_i,h-w(h) \rangle \geq 0$ for every $w\in W$, which implies that
	$$\langle \omega_i,h' \rangle = \sum_{w\in W} \lambda_w\langle \omega_i,w(h) \rangle \leq \sum_{w\in W} \lambda_w \langle \omega_i,h \rangle=\langle \omega_i,h \rangle.$$
	Hence, $\Xi_h$ lies on one side of $H^{(i)}$ and $\Xi_h\cap H^{(i)}$ is an exposed face. Moreover, $H^{(i)}$ contains $h$ and $s_j(h)$ whenever $j\neq i$ since
	$$\langle \omega_i,s_j(h)\rangle=\langle \omega_i,h\rangle-\langle \alpha_j,h\rangle \langle\omega_i,\alpha_j^\vee\rangle=\langle \omega_i,h\rangle.$$
	As these $n$ points are affinely independent, it follows that $\Xi_h\cap H^{(i)}$ has dimension $n-1$ and is therefore a facet of $\Xi_h$. Iterating the computation above shows that $H^{(i)}$ contains $w(h)$ for every $w\in W^{(i)}$ and hence $F^{(i)}\subset \Xi_h\cap H^{(i)}$. Conversely, if $w\notin W^{(i)}$, then every reduced decomposition of $w$ contains $s_i$, i.e. $i\in\op{supp}(w)$. By Lemma \ref{lem:weyl_difference}, this implies that $h-w(h)$ is a linear combination of simple coroots in which $\alpha_i^\vee$ occurs with a positive coefficient, so that $\langle \omega_i,w(h) \rangle<\langle \omega_i,h \rangle$. Let now $h'\in \Xi_h\cap H^{(i)}$, then $h'$ has to be a convex combination of the points in $W\cdot h\cap H^{(i)}$ which we have just shown to coincide with $W^{(i)}\cdot h$. Thus, it follows that $\Xi_h\cap H^{(i)}\subset F^{(i)}$ and the proposition is proven.
\end{proof}

\noindent Observe that we consider the hyperplanes $H^{(1)},\ldots,H^{(n)}$ as subsets of $E_h$, which is the affine span of $\Xi_h$, rather than the full vector space $\a$. If $d>n$, then for every $i\in I$ there are many hyperplanes in $\a$ that intersect $E_h$ in $H^{(i)}$, but this extension process is irrelevant to our discussion since we are only interested in intersections with $\Xi_h\subset E_h$. As a subset of $E_h$, $H^{(i)}$ is uniquely determined by $F^{(i)}$. Since the set of exposed faces is $W$-invariant and closed under intersections, the following corollary follows immediately from the previous proposition.

\begin{cor} \label{cor:CoxeterFaces}
	For every subset $J\subset I$, the set $F_J:=\op{conv}(W_J\cdot h)$ is an exposed face of $\Xi_h$ of dimension $|J|$. More generally, for every $w\in W$, the set $w(F_J)$ is an exposed face of dimension $|J|$ and coincides with $\op{conv}(wW_J\cdot h)$. In particular, every $W$-translate of $h$ is a vertex of $\Xi_h$.
\end{cor}

\noindent If $J,J'\subset I$ are two distinct subsets of $I$, then the affine spans of the Weyl group orbits $W_J\cdot h$ and $W_{J'}\cdot h$ do not coincide, which implies that the faces $F_J$ and $F_{J'}$ are also distinct. Hence, $h$ is contained in at least $2^n$ faces of $\Xi_h$. More generally, we have $w(F_J)=w'(F_{J'})$ if and only if $J=J'$ and $w^{-1}w'\in W_J$, so the map $wW_J\mapsto w(F_J)$ defines an order-preserving injection $\Sigma^\ast(W,S)\to\mc{F}(\Xi_h)$. In order to show that this map is also surjective, we need more information on the topology of $\Xi_h$. 

\begin{lem} \label{topology_lem}
	Let $h\in C$ and $h'\in \Xi_h$ be such that $v:=h-h'=\sum_{i\in I} c_i\alpha_i^\vee$ satisfies $c_i>0$ for all $i\in I$. Then the open line segment $(h,h')$ is contained in the relative interior of $\Xi_h\subset E_h$.
\end{lem}

\begin{proof}
	Choose any $\lambda\in (0,1)$ and set $x:=h-\lambda v\in \Xi_h$. First, we show that for every $i\in I$ there exists some $\mu_i>0$ such that $x-\mu_i\alpha_i^\vee\in\Xi_h$ by writing this point as a convex combination of $\{h,h',s_i(h)\}$. To this end, choose $\mu_i>0$ such that $c:=\frac{\mu_i}{\langle \alpha_i,h\rangle}<1-\lambda$, set $b:=\lambda$ and $a:=1-b-c$. Then we have $a,b,c\in (0,1)$ and $a+b+c=1$. Moreover, the following computation shows that the corresponding convex combination is the desired point:
	\begin{align*}
		ah+bh'+cs_i(h) &= ah+bh+b(h'-h)+c(h-\langle \alpha_i,h\rangle\alpha_i^\vee) \\
		&= (a+b+c)h-bv-c\langle \alpha_i,h\rangle\alpha_i^\vee \\
		&= h-\lambda v-\mu_i\alpha_i^\vee=x-\mu_i\alpha_i^\vee
	\end{align*}
	On the other hand, $\Xi_h$ also contains $x+\mu v$ for all $0\leq \mu\leq \lambda$. Since $\lambda>0$ and $c_i>0$ for all $i\in I$, it follows that for every $i\in I$, $\Xi_h$ also contains $x+\mu_i'\alpha_i^\vee$ for sufficiently small $\mu_i'>0$. Altogether, we obtain that $\Xi_h$ contains the cube
	$$x+\prod_{i\in I} [-\mu_i,\mu_i']\alpha_i^\vee\subset E_h.$$
	Since the simple roots form a basis of $\a'$, this set has non-empty relative interior in $E_h$, so $x$ lies in the relative interior of $\Xi_h\subset E_h$.
\end{proof}

\begin{prop} \label{prop:closedHull}
	If $h\in C$, then $\Xi_h$ is closed.
\end{prop}

\begin{proof}
	If $W$ is finite, then $\Xi_h$ is a convex polytope and hence compact. Thus, for the rest of the proof, we may assume that $|W|=\infty$. For every $k\in\N_0$, we set 
	$$U_k:=\{w\in W: l(w)\leq k\} \qquad \text{and} \qquad X_k:=\op{conv}(U_k\cdot h).$$
	Since $U_k$ is finite, $X_k$ is a convex polytope and hence compact. Moreover, these polytopes yield a filtration of $\Xi_h$ in the sense that $X_k\subset X_{k+1}$ for all $k\in\N_0$ and $\Xi_h=\bigcup_{k=0}^\infty X_k$. Let $(y_m)_{m=1}^\infty$ be a sequence in $\Xi_h$ converging to some $y\in E_h$. We are going to show that there exists some $k'\in\N_0$ such that $y_m\in X_{k'}$ for all $m\in \N$. Once this is established, we may conclude that the sequence $(y_m)_{m=1}^\infty$ is contained in the compact set $X_{k'}$, so that $y\in X_{k'}\subset\Xi_h$. The proof of this assertion will be reduced to a sequence of claims. Before we formulate them, we need to introduce some notation. Throughout this proof, we will identify $E_h$ with $\R^n$ by moving $h$ to $0$ and mapping $\Phi_0^\vee$ to the standard basis of $\R^n$ in the natural order. For every $i\in I$, $\omega_i-\langle\omega_i,h\rangle$ corresponds to the $i$-th coordinate functional under this identification. Thus, if we set $\omega:=\sum_{i\in I} \omega_i$, then
		\begin{equation}
			|\langle \omega,p \rangle-\langle \omega,q\rangle|=\bigg| \sum_{i=1}^n (p_i-q_i)\bigg|\leq \sum_{i=1}^n |p_i-q_i|=||p-q||_1, \qquad p,q\in E_h\cong\R^n, \label{l1Norm}
		\end{equation}
	where $||\cdot||_1$ denotes the $\ell^1$-norm on $\R^n$. Since all norms on a finite-dimensional vector space are equivalent, the $\ell^1$-topology on $\R^n$ coincides with the standard Euclidean topology. Next, set $U_{-1}:=\emptyset$ and define for every $k\in\N_0$
	\begin{align*}
		V_k & := U_k\setminus U_{k-1} = \{w\in W:l(w)=k\}, \\
		m_+(k) &:= \max_{w\in V_k} \langle\omega,w(h) \rangle, \\
		m_-(k) &:= \min_{w\in V_k} \langle\omega,w(h) \rangle, \\
		d(k) &:= m_+(k)-m_-(k).
	\end{align*}
	Since we assume that $W$ is infinite, we have $V_k\neq\emptyset$ for all $k\in\N_0$. We set $c:=\min_{i\in I} \langle \alpha_i,h \rangle>0$.
	\medskip \newline
	\textbf{Claim 1:} For every $w\in W$ and $i\in I$ with $l(w)<l(s_iw)$, we have $\langle \omega,s_iw(h) \rangle\leq \langle \omega,w(h) \rangle-c$.
	\medskip \newline
	\textit{Proof of Claim 1:} By \eqref{PositivityCriterion}, we have $\langle \alpha_i,w(h)\rangle=\langle w^{-1}(\alpha_i),h\rangle \geq c$ if $l(w)<l(s_iw)$, so that
		$$\langle \omega,s_iw(h) \rangle=\langle \omega,w(h) \rangle-\langle \alpha_i,w(h)\rangle\langle \omega,\alpha_i^\vee \rangle=\langle \omega,w(h) \rangle-\langle \alpha_i,w(h)\rangle\leq \langle \omega,w(h) \rangle-c.$$
	\textbf{Claim 2:} For every $k\in\N_0$ there exists some $k'\in\N$ such that $m_+(k')<m_-(k)$.
	\medskip \newline 
	\textit{Proof of Claim 2:} Let $k'\geq k$ and $w\in V_{k'}$ be arbitrary and write $w=w_1w_2$ with $l(w_1)=k'-k$ and $l(w_2)=k$. Iterating the computation in Claim 1 shows that
	$$\langle \omega,w(h)\rangle\leq \langle \omega,w_2(h)\rangle-(k'-k)c\leq m_+(k)-(k'-k)c=m_-(k)+d(k)-(k'-k)c.$$
	Since the right-hand side does not depend on $w$, we conclude that 
	$$m_+(k')\leq m_-(k)+d(k)-(k'-k)c.$$ 
	If we choose $k'$ so large that $(k'-k)c>d(k)$, the claim is satisfied.
	\medskip \newline
	\textbf{Claim 3:} For every $k\in\N_0$, every proper face $F$ of $X_k$ is either contained in a proper face of $\Xi_h$ or it is of the form $F=\op{conv}(U\cdot h)$ for a subset $U\subset V_k$.
	\medskip \newline
	\textit{Proof of Claim 3:} The statement is trivial if $k=0$ since $X_0=\{h\}$ is a vertex of $\Xi_h$. Thus, we may assume that $k\geq 1$. Since $X_k=\conv(U_k\cdot h)$ is a convex polytope, every face is exposed. Hence, $F$ is of the form $F=\op{conv}(U\cdot h)$ for some proper subset $U\subsetneq U_k$. If $U$ is contained in a non-trivial standard coset of $W$, i.e. $U\subset wW_J$ for some $w\in W$ and $J\subsetneq I$, then $F$ is contained in the proper face $w(F_J)=\op{conv}(wW_J\cdot h)$ of $\Xi_h$. Otherwise, let $u\in U$ be arbitrary, then for every $i\in I$, there exists some $u_i\in U$ such that $u_i\notin uW^{(i)}$. Equivalently, $u^{-1}u_i\notin W^{(i)}$ and hence $s_i\in S(u^{-1}u_i)$. Now Lemma \ref{lem:weyl_difference} implies that $h-u^{-1}u_i(h)$ is a sum of simple coroots in which all coefficients are non-negative and the coefficient of $\alpha_i^\vee$ is non-zero. Now consider the point
	$$h':=\sum_{i=1}^n \frac{1}{n} u^{-1}u_i(h)=u^{-1}\bigg(\sum_{i=1}^n \frac{1}{n}u_i(h)\bigg)\in u^{-1}(F).$$
	Then we have
	$$h-h'=\sum_{i=1}^n \frac{1}{n}(h-u^{-1}u_i(h)),$$
	so that $h-h'$ is a sum of simple coroots in which all coefficients are positive. By Lemma \ref{topology_lem}, the open line segment $(h,h')$ is contained in the relative interior of $\Xi_h$. More precisely, if $l(u)<k$, the proof of Lemma \ref{topology_lem} actually shows that this line runs through the relative interior of $u^{-1}(X_k)$. Indeed, observe that in that proof we have only used $h$, $h'$ and the simple reflections to create open neighbourhoods and $s_i(h)=u^{-1}us_i(h)\in u^{-1}(X_k)$ for every $i\in I$ since $l(us_i)\leq l(u)+1\leq k$. But then the line segment $(u(h),u(h'))\subset F$ is contained in the relative interior of $X_k$, contradicting the assumption that $F$ is a proper face of $X_k$. Since $u\in U$ was arbitrary, $F$ can only be a face if $l(u)=k$ for all $u\in U$, i.e. $U\subset V_k$.
	\medskip \newline
	\textbf{Claim 4:} If $q\in\Xi_h\setminus X_k$ for some $k\in\N_0$, then $\langle\omega,q\rangle<m_+(k)$.
	\medskip \newline
	\textit{Proof of Claim 4:} The claim is clear if $k=0$ since $X_0=\{h\}$ and every point $q\in\Xi_h\setminus\{h\}$ satisfies $\langle\omega,q\rangle<\langle\omega,h\rangle=m_+(0)$, so we may assume that $k\geq 1$. Consider the line segment $L:=[h,q]\subset\Xi_h$ connecting $h$ and $q$. Since $h\in X_0\subset X_k$ and $q\notin X_k$, there exists a supporting hyperplane $E$ for $X_k$ such that $q$ and $h$ are not contained in the same closed half-space defined by $E$ and $L$ intersects the corresponding face $F:=X_k\cap E$. If $F$ was contained in a face of $\Xi_h$, then we could choose $E$ to be a supporting hyperplane for $\Xi_h$, but then $q\in\Xi_h$ would imply that $q$ and $h$ were both contained in the same closed half-space. Thus, $F$ is of the form $F=\op{conv}(U\cdot h)$ for some $U\subset V_k$ by Claim 3. As $k\geq 1$ and $q\notin F$, this face neither contains $h$ nor $q$, so $h$ and $q$ lie in different open half-spaces defined by $E$ and the line $L$ intersects $F$ in an interior point $x\in (h,q)$. Since $\langle\omega,u(h)\rangle\leq m_+(k)$ for all $u\in U$ and $x$ is a convex combination of these vertices, we have $\langle\omega,x\rangle\leq m_+(k)$. In addition, since $k\geq 1$, we also have $m_+(k)<\langle \omega,h \rangle$. Hence, $\omega$ is strictly monotonically decreasing along $L$. As $x$ lies strictly between $h$ and $q$, this yields that $\langle\omega,q\rangle<\langle\omega,x\rangle\leq m_+(k)$.
	\medskip \newline	
	\textbf{Claim 5:} Let $k\in\N_0$ and choose $k'>k$ such that $\delta:=m_-(k)-m_+(k')>0$. Then $||p-q||_1> \delta$ for all $p\in X_k$ and $q\in \Xi_h\setminus X_{k'}$.
	\medskip \newline 
	\textit{Proof of Claim 5:} On the one hand, $p\in X_k$ is a convex combination of $W$-translates of $h$ of length at most $k$. The value of $\omega$ on each of these vertices is at least $m_-(k)$ since $m_-$ is monotonically decreasing by Claim 1. This implies $\langle\omega,p\rangle\geq m_-(k)$. On the other hand, we have $\langle\omega,q\rangle<m_+(k')$ by Claim 4, so \eqref{l1Norm} implies that
	$$||p-q||_1\geq |\langle \omega,p \rangle-\langle \omega,q \rangle|>m_-(k)-m_+(k')=\delta.$$
	\textit{Conclusion of the proof:} Since a convergent sequence is bounded, we find $R>0$ such that $(y_m)_{m=1}^\infty$ is contained in a closed ball of radius $R$ (w.r.t. the $\ell^1$-norm) around $y_1$. In addition, since for every $m\in\N$, $y_m$ is a finite convex combination of $W$-translates of $h$, there exists some $k(m)\in\N_0$ such that $y_m\in X_{k(m)}$. Set $k:=k(1)$ and choose $k'>k$ such that $\delta:=m_-(k)-m_+(k')\geq R$. Since $||y_1-y_m||_1\leq R$ for all $m\in\N$, Claim 5 implies that the sequence $(y_m)_{m=1}^\infty$ must be contained in $X_{k'}$, as required.
\end{proof}

\noindent With these topological results at hand, we are now in a position to show that the exposed faces of $\Xi_h$ that we constructed so far actually make up the entire face poset.

\begin{prop} \label{prop:exposedFaces}
	If $h\in C$ and $F$ is a face of $\Xi_h$, then $F=w(F_J)$ for some $w\in W$ and a unique subset $J\subset I$. In particular, every face of $\Xi_h$ is exposed.
\end{prop}

\begin{proof}
	We have already discussed uniqueness of $J$ after Corollary \ref{cor:CoxeterFaces}, so it remains to show existence. If $F=\Xi_h$, the claim is satisfied if and only if $J=I$ and $w\in W$ is arbitrary. Thus, we may assume that $F$ is a proper face of $\Xi_h$. We will argue by induction on $n=|I|$. If $n=1$, then $\Xi_h$ is the line segment connecting the two distinct points $h$ and $s_1(h)$. The only proper faces of $\Xi_h$ are those two points which have the required description for $J=\emptyset$. \\
	Let now $n>1$ be arbitrary and let $x$ be a point in the relative interior of $F$. As $F$ is a proper face, the point $x$ lies on the boundary of $\Xi_h$ and since $\Xi_h$ is closed, it is contained in a non-trivial supporting hyperplane $E$ of $\Xi_h$. Let $G:=\Xi_h\cap E$ be the corresponding exposed face, then $G$ is also proper and $F\subset G$ because $G$ contains a point from the relative interior of $F$. This exposed face is of the form $G=\op{conv}(W'\cdot h)$ for some proper subset $W'\subset W$ and upon replacing $F$ by $w^{-1}(F)$ and $G$ by $w^{-1}(G)$ for a suitable $w\in W$, we may assume that $1\in W'$ and $h\in G$. We will show that there exists some $i\in I$ such that $W'\subset W^{(i)}$. \\
	Assume for a contradiction that for every $i\in I$, there exists some $w_i\in W'$ with $i\in\op{supp}(w_i)$. Then it follows from Lemma \ref{lem:weyl_difference} that $h-w_i(h)$ is a linear combination of simple coroots in which $\alpha_i^\vee$ occurs with a positive coefficient. Since all other coefficients are non-negative, we may form a suitable convex combination of these points to deduce that there exists a point $y\in G$ such that $h-y=\sum_{i\in I} c_i\alpha_i^\vee$ with $c_i>0$ for all $i\in I$. Now Lemma \ref{topology_lem} implies that the line connecting $h$ and $y$, which is contained in $G$, runs through the interior of $\Xi_h$, a contradiction. \\
	Hence, we must have $W'\subset W^{(i)}$ for some $i\in I$, which implies that $G$ is contained in the facet $F^{(i)}$. The same is then true for $F$, so $F$ is a face of $F^{(i)}$. The induction hypothesis now implies that there exists a unique (not necessarily proper) subset $J\subset I\setminus\{i\}$ and some $w\in W^{(i)}$ such that $F=w(F_J)=\op{conv}(wW_J\cdot h)$. Since we have already shown that every face of this form is exposed, the proof is finished.
\end{proof}

\noindent The following theorem is now an immediate consequence of Corollary \ref{cor:CoxeterFaces} and Proposition \ref{prop:exposedFaces} and provides a completely combinatorial description of the face poset of $\Xi_h$. 

\begin{thm} \label{thm:dualcoxeter_iso}
	If $h\in C$, then the faces of $\Xi_h$ correspond bijectively to the standard cosets of $(W,S)$. More precisely, for every standard coset $wW_J\subset W$ with $w\in W$ and $J\subset I$, the set $w(F_J)=\conv(wW_J \cdot h)$ is a face of $\Xi_h$ of dimension $|J|$ and, conversely, every face is of this form for a unique standard coset of $(W,S)$. In particular, the poset $\mc{F}(\Xi_h)$ is isomorphic to $\Sigma^*(W,S)$.
\end{thm}

\begin{rmk} \label{rmk:CellDecomp}
	This theorem shows that every vertex of $\Xi_h$, i.e. every $W$-translate of $h$, is contained in precisely $2^n$ faces of $\Xi_h$. In particular, every face is only contained in finitely many other faces. As every face is an intersection of sufficiently many facets, this means that every point of $\Xi_h$ is only contained in finitely many defining hyperplanes, so the hyperplane arrangement
	$$\{w(H^{(i)}) \mid w\in W,\, i\in I\},$$
	whose intersection pattern induces the facial structure on $\Xi_h$, is \textit{locally finite}. Note that these hyperplanes are not linear. In the language of Coxeter groups, this is reflected in the fact that for every standard coset $wW_J$, there exist only finitely many standard cosets of $(W,S)$ containing it. In contrast, a face may well contain infinitely many other faces. In fact, for every $J\subset I$, the face poset $\mc{F}(F_J)$ is isomorphic to $\Sigma^\ast(W_J,S_J)$, where $S_J=\{s_j \mid j\in J\}$, which is finite if and only if the standard subgroup $W_J\subset W$ is finite. Since the generating set $S$ is finite, this is the case if and only if the length function of $W$ is bounded on $W_J$. By inspecting the proof of Proposition \ref{prop:closedHull}, this implies that the face $F_J$ is bounded (and hence, compact) if and only if $W_J$ is finite, so the bounded faces of $\Xi_h$ correspond precisely to the finite parabolic subgroups of $W$. Trivially, vertices and edges contain only finitely many faces, so unbounded faces must be at least two-dimensional. In combinatorial terms, this reflects the fact that there is only one Coxeter group of rank $1$, namely the finite group $\Z_2$. \\
	The group $W$ is said to be \textit{k-spherical} for some $2\leq k\leq n$ if every standard subgroup $W_J$ with $|J|\leq k$ is finite. In this case, every face of dimension at most $k$ has only finitely many faces or, equivalently, is bounded. The only infinite Coxeter group of rank $2$ is the infinite dihedral group
	$$D_\infty=\langle s,t \mid s^2=t^2=1 \rangle,$$
	so $W$ is 2-spherical if and only if the order of $s_is_j\in W$ is finite for all $i\neq j\in I$. By \eqref{CoxeterTable}, this is the case if and only if the GCM $\A$ satisfies $a_{ij}a_{ji}\leq 3$ for all $i\neq j\in I$. If $\A$ is of affine type, then $W$ is infinite but every proper standard subgroup of $W$ is finite. In this case, $W$ is $(n-1)$-spherical, so every proper face of $\Xi_h$ is bounded, but $\Xi_h$ itself is not. \\
	This behaviour is in a sense ``dual'' to the cellular structure of the Tits cone. A \textit{cell} of $X$ is any set of the form $w(D_J)\subset X$ for some $w\in W$ and $J\subset I$, where 
		$$D_J:=\{h'\in\a: \langle\alpha_i,h'\rangle=0\,\,\forall i\in J,\,\, \langle\alpha_i,h'\rangle>0\,\,\forall i\in I\setminus J\}.$$
	The sets $\{D_J \mid J\subset I\}$ are precisely the faces of the closed fundamental chamber $\overline{C}$ and for $h'\in w(D_J)$, the stabilizer of $h'$ is the parabolic subgroup $wW_Jw^{-1}\subset W$. In particular, it depends only on the cell $w(D_J)$, so the cell complex of $X$ is isomorphic to the ``usual'' Coxeter complex $\Sigma(W,S)$ where standard cosets are ordered by reversed inclusion. The Tits cone thus provides a geometric realization of $\Sigma(W,S)$. Every cell of $X$ contains only finitely many cells, but may well be contained in infinitely many different ones. Observe that the cellular structure on $X$ is induced by the intersection pattern of the hyperplane arrangement
	$$\{w(\op{ker}(\alpha_i)) \mid w\in W,\, i\in I\}.$$
	In contrast to the situation of $\Xi_h$, all these hyperplanes are linear, which implies that this hyperplane arrangement is locally finite if and only if it is finite or, equivalently, $W$ is finite.
\end{rmk}

\noindent Another useful result, which is a direct consequence of Theorem \ref{thm:dualcoxeter_iso}, is the following. We use standard terminology of Coxeter complexes as described in \cite[Chapter 3]{AB}.

\begin{cor} \label{cor:graphAdjacency}
	If $h\in C$, the underlying graph of $\Xi_h$, i.e. the set of all edges and vertices of $\Xi_h$, is isomorphic to the chamber system of the Coxeter complex $\Sigma(W,S)$. In other words, two vertices $w_1(h)$ and $w_2(h)$ of $\Xi_h$ are joined by an edge if and only if the chambers $\{w_1\}$ and $\{w_2\}$ of $\Sigma(W,S)$ are adjacent, i.e. $w_2=w_1s_i$ for some $i\in I$. Under this isomorphism, the graph distance in $\Xi_h$ corresponds to the chamber distance in $\Sigma(W,S)$ and minimal paths between vertices of $\Xi_h$ are in bijective correspondence with minimal galleries between the corresponding chambers. Equivalently, the underlying graph of $\Xi_h$ is also isomorphic to the Cayley graph of $(W,S)$.
\end{cor}

\noindent Finally, recall that a closed convex set can be written as the intersection of all supporting half-spaces. As every proper face of $\Xi_h$ is contained in a facet, $\Xi_h$ can also be reconstructed as the intersection of all defining half-spaces. This observation yields the next corollary which will play a crucial role in the proof of Theorem \ref{thm:KMKostant}. 

\begin{cor}\label{cor:inequality1}
	If $h\in C$, then a point $h'\in\a$ is contained in $\Xi_h$ if and only if it satisfies the following (in-)equalities:
		$$\begin{cases}
			\langle \omega_i,w(h')\rangle\leq \langle\omega_i,h \rangle & \text{for all $w\in W$ and $i\in I$} \\
			\langle \phi_i,h'\rangle=\langle\phi_i,h \rangle & \text{for all $i\in \{n+1,\ldots,d\}$}
		\end{cases}$$
\end{cor}

\begin{proof}
	Since the set of defining hyperplanes of $\Xi_h$ is $\{w^{-1}(H^{(i)}) \mid w\in W,\, i\in I\}$ and $\Xi_h$ is the intersection of all defining half-spaces, which we view as subsets of $E_h$, $h'$ is contained in $\Xi_h$ if and only if it lies in $E_h$ and satisfies the corresponding inequalities $\langle \omega_i, w(h') \rangle\leq\langle \omega_i, h\rangle$ for all $w\in W$ and $i\in I$. Moreover, the condition $h'\in E_h$ is equivalent to $h-h'\in\a'$ and hence to $\langle \phi_i,h \rangle=\langle \phi_i,h' \rangle$ for all $i\in \{n+1,\ldots,d\}$.
\end{proof}

\noindent It is also not difficult to show that the inequalities in Corollary \ref{cor:inequality1} directly imply that $h'$ lies in the Tits cone $X\subset\a$. Once this is established, it turns out that most inequalities are in a certain sense redundant. In fact, if $h'$ is already contained in $\overline{C}$, then $\langle \omega_i, w(h') \rangle\leq\langle \omega_i, h' \rangle$ for all $w\in W$ and $i\in I$ by \eqref{fundamentalChamber}, so that the inequalities with $w\neq 1$ are implied by those with $w=1$. If $h'\in X$ does not necessarily lie in $\overline{C}$, then $h'$ is contained in $\Xi_h$ if and only if its $W$-translate $h''\in (W\cdot h')\cap\overline{C}$ has this property. As we have just seen, this is the case if and only if it lies in $E_h$ and satisfies $\langle \omega_i,h''\rangle\leq \langle\omega_i,h \rangle$ for all $i\in I$. 
\medskip \newline
We have now fully determined the face poset of $\Xi_h$ and found necessary and sufficient conditions that decide whether a point in $\a$ is contained in $\Xi_h$ or not. Next, we want to study the intersections of $\Xi_h$ with parallel translates of the defining hyperplanes. As we shall see, these intersections are invariant under subgroups of $W$ of rank $n-1$ which will allow us to apply inductive arguments in the proof of Theorem \ref{thm:KMKostant}.
\medskip \newline
For the rest of this section, we fix an arbitrary $i\in I$ and consider parallel translates of the defining hyperplane $H^{(i)}$. Given any point $y\in E_h$, we set $y_i:=\langle \omega_i,y \rangle$. For every $t\in\R$, we define
$$H^{(i)}(t):=\{y\in E_h: y_i=t\},$$
so that $H^{(i)}(t)$ is parallel to $H^{(i)}=H^{(i)}(h_i)$. We have already seen that every point $y\in\Xi_h$ satisfies $y_i\leq h_i$ with equality if and only if $y\in F^{(i)}$. Thus, the intersection
$$P^{(i)}(t):=\Xi_h\cap H^{(i)}(t)$$
can only be non-empty if $t\leq h_i$. The existence of a lower bound depends on whether the index of $W^{(i)}$ in $W$ is finite. This is obviously the case if $W$ is finite. If $W$ is infinite, however, this is only possible in a rather trivial situation.

\begin{thm} \label{thm:FiniteIndexSubgroup}
	Let $(W,S)$ be a Coxeter system, let $I_1,\ldots,I_m\subset I$ be the connected components of the Coxeter diagram of $(W,S)$ and set 
	$$I_\infty:=\bigcup \{I_k \mid \text{$1\leq k\leq m,\, W_{I_k}$ is infinite}\}.$$
	Then $W_{I_\infty}$ is the smallest finite-index standard subgroup of $W$. If $W_J$ is a finite-index standard subgroup of $W$, then $I_\infty\subset J$ and $W_J\cong W_{I_\infty}\times W_{J\setminus I_\infty}$. In particular, if $(W,S)$ is infinite and irreducible, then $W$ does not contain a non-trivial standard subgroup of finite index.
\end{thm}

\begin{proof}
	\cite[Corollary 3.3 and 3.4]{Hosaka}
\end{proof}

\noindent In the formulation of the following lemma, we use the fact that in every Coxeter system $(W,S)$ with $|W|<\infty$, there exists a unique element $w_+\in W$ such that $l(w_+)$ is maximal.\footnote{In the literature on Coxeter groups, it is customary to denote this element by $w_0$ rather than $w_+$. We do not use that notation on purpose since it will sometimes be convenient to be able to list elements of $W$ as $w_0,w_1,w_2,\ldots$, where the first element should not be confused with $w_+$ if $W$ is finite.} It is of order two and conjugation by $w_+=w_+^{-1}$ preserves the generating set $S$. Hence, it induces an automorphism of the Coxeter diagram of $(W,S)$. It is called the \textit{longest word} of $W$.

\begin{lem} \label{lem:IndexDistinction} \hfill
	\begin{itemize} 
		\item[(i)] If $[W:W^{(i)}]$ is finite, let $w_+$ be the longest word of $W_{I\setminus I_\infty}$ and let $j\in I$ be the unique index such that $s_j=w_+s_iw_+$. Then $P^{(i)}(t)$ is non-empty if and only if $t\in [-h_j,h_i]$.
		\item[(ii)] If $[W:W^{(i)}]$ is infinite, then $P^{(i)}(t)$ is non-empty for all $t\in (-\infty,h_i]$.
	\end{itemize}
\end{lem}

\begin{proof} \hfill
	\begin{itemize}
		\item[(i)] Since $W=W_{I_\infty}\times W_{I\setminus I_\infty}$ is the direct product of an infinite and a finite Coxeter group, conjugation by the longest word of $W_{I\setminus I_\infty}$ defines an automorphism $\sigma$ of the Coxeter diagram of $(W,S)$ which fixes all vertices in $I_\infty$. Hence, we have $w_+s_kw_+=s_{\sigma(k)}$ for all $k\in I$, where $\sigma(k)=k$ for all $k\in I_\infty$. This implies that $w_+(\alpha_k^\vee)=\pm\alpha_{\sigma(k)}^\vee$ for all $k\in I$, where the sign is positive if and only if $k\in I_\infty$ by \eqref{PositivityCriterion} because $l(w_+)<l(w_+s_k)$ if and only if $k\in I_\infty$. Thus, the action of $w_+$ on $\a^\ast$ fixes $\omega_k$ for every $k\in I_\infty$ and for $k\in I\setminus I_\infty$, it maps $\omega_k$ to $-\omega_{\sigma(k)}$. If $[W:W^{(i)}]$ is finite, then $i\notin I_\infty$ and $w_+(\omega_i)=-\omega_j$ for $j=\sigma(i)\in I\setminus I_\infty$. Since $i\in\op{supp}(w_+)$, we obtain from Lemma \ref{lem:weyl_difference} that
			$$0<\langle\omega_i,h-w_+(h)\rangle=h_i-\langle w_+(\omega_i),h\rangle=h_i+h_j,$$
		so the interval $[-h_j,h_i]$ is well-defined. If $h'\in\Xi_h$, then by $W$-invariance of $\Xi_h$, we also have $w_+(h')\in\Xi_h$. Proposition \ref{prop:maxface} now implies that $\langle \omega_k,w_+(h')\rangle\leq h_k$ for all $k\in I$ and hence
		$$h_i'=\langle \omega_i,h' \rangle=\langle w_+(\omega_i),w_+(h')\rangle=\langle -\omega_j,w_+(h')\rangle\geq -h_j.$$
		Thus, in this case $H^{(i)}(t)$ can only intersect $\Xi_h$ if $t\geq -h_j$. On the other hand, we know that $P^{(i)}(h_i)=F^{(i)}$ is a face of $\Xi_h$ and, using the same reasoning, one verifies that
		$$P^{(i)}(-h_j)=\op{conv}(W^{(i)}w_+\cdot h)=\op{conv}(w_+W^{(j)}\cdot h)=w_+(F^{(j)})$$
		is also a face of $\Xi_h$ and thus non-empty. Since $\Xi_h$ is convex, it follows that $P^{(i)}(t)$ is non-empty for all $t\in [-h_j,h_i]$.
		\item[(ii)] Consider the orbit of $\omega_i$ under the $W$-action on $\a^\ast$. Since $\langle\omega_i,\alpha_j^\vee\rangle=\delta_{ij}$ for all $i,j\in I$, it follows easily by induction on the word length (cf. the proof of \cite[Proposition 3.12 a)]{Kac}) that the stabilizer of $\omega_i$ is the standard subgroup $W^{(i)}$. Since $[W:W^{(i)}]=\infty$, the $W$-orbit of $\omega_i$ is infinite. In addition, it is contained in $\omega_i-Q_+$ which implies that the orbit contains elements of the form $\omega_i-\alpha$ with $\alpha\in Q_+$ and $\op{ht}(\alpha)$ unbounded. For every $t\in\R$, we therefore find $w\in W$ such that $t_0:=\langle \omega_i,w(h) \rangle\leq t$. Since $P^{(i)}(t_0)$ contains $w(h)$, it is non-empty. By convexity, it follows that $P^{(i)}(t)$ is non-empty for all $t\in [t_0,h_i]$. As $t$ can be chosen arbitrarily small and $t_0\leq t$, the assertion follows. \qedhere
	\end{itemize}
\end{proof}

\noindent Using the notation of the preceding lemma, we define an interval $I_h\subset\R$ by
$$I_h:=\begin{cases} [-h_j,h_i] & \text{if $[W:W^{(i)}]<\infty$,} \\ (-\infty,h_i] & \text{if $[W:W^{(i)}]=\infty$.}\end{cases}$$
so that we may express the previous result result more briefly as
\begin{equation}
	\Xi_h=\bigcup_{t\in I_h} P^{(i)}(t). \label{XiFoliation}
\end{equation}

\noindent Our next goal is to determine the face poset of $P^{(i)}(t)$ for $t\in I_h$. By the discussion in Section \ref{ConvexGeom}, for every face $F$ of $P^{(i)}(t)$ there exists a unique minimal face $F'$ of $\Xi_h$ such that $F=F'\cap H^{(i)}(t)$ and its dimension is determined by \eqref{FaceIntersection}. In particular, every face of $P^{(i)}(t)$ is exposed and every vertex of $F$ is either a vertex of $F'$ that happens to be contained in $H^{(i)}(t)$, or lies in the interior of a unique edge of $F'$. The first case is only possible if $t=\langle \omega_i,w(h) \rangle$ for some $w\in W$. Thus, in general $F$ need not contain any vertices of $\Xi_h$, but it still inherits the following property from $F'$.

\begin{lem} \label{lem:face_hull_lemma}
	Every face of $P^{(i)}(t)$ is the convex hull of its vertices. 
\end{lem}

\noindent If $P^{(i)}(t)$ is a convex polytope, this is of course a well-known property. For arbitrary closed convex sets, however, this is not necessarily true without further requirements on the set. In our situation, the assertion can be easily reduced to the compact case by a cut-off argument.

\begin{proof}
	Let $F$ be any face of $P^{(i)}(t)$ and choose $F'$ as the smallest face of $\Xi_h$ such that $F=F'\cap H^{(i)}(t)$. By the preceding discussion, the set $Y$ of vertices of $F$ is discrete and not empty and $\op{conv}(Y)\subset F$ holds by convexity. To show the converse, we claim that for every $x\in F$, there exists a finite subset $Y'\subset Y$ such that $x\in\conv(Y')$. Set $\omega:=\sum_{i\in I}\omega_i\in\a^\ast$, let $x\in F$ be arbitrary and consider the closed half-space
	$$E_+:=\{y\in E_h: \langle\omega,x\rangle\leq \langle\omega,y\rangle \}.$$
	Then $F_+:=F\cap E_+$ is compact since $F=F'\cap P^{(i)}(t)$ is closed and each of the linearly independent functionals $\omega_1,\ldots,\omega_n\in\a^\ast$ is bounded on $F_+$. Every face of $F_+$ is exposed, so by \cite[Corollary 18.5.1]{Rockafellar}, $F_+$ is the convex hull of its vertices. Now every vertex of $F_+$ is either a vertex of $F$ contained in $E_+$, or it lies in the interior of a unique edge of $F$. Let $Y_1'\subset Y$ be the set of vertices of $F$ contained in $E_+$ and let $Y_2'\subset Y$ be the set of all endpoints of edges of $F$ which start in $E_+$ and end outside of it. Since $Y$ is discrete and $F_+$ is compact, the set $Y':=Y_1'\cup Y_2'$ is finite and every vertex of $F_+$ is either contained in $Y'$ or it is a convex combination of two points in $Y'$. Hence, it follows that $x\in F_+\subset\conv(Y')\subset\conv(Y)$, as claimed. As $x\in F$ was arbitrary, we conclude that $F\subset\conv(Y)$.
\end{proof}

\noindent Moreover, $P^{(i)}(t)$ is invariant under the action of the rank-$(n-1)$ subgroup $W^{(i)}\subset W$ which is the Weyl group of the submatrix $\A^{(i)}$ of $\A$ with the $i$-th row and column removed. The set $\Phi^{(i)}$ of real roots of $\A^{(i)}$ can be identified with the subset of $\Phi$ that lies in the $\Z$-span of $\Phi_0\setminus\{\alpha_i\}$. The Kac-Moody root datum $\mc{D}$ with GCM $\A$ can be restricted to a free and cofree root datum with GCM $\A^{(i)}$ without changing $\Lambda$ which allows us to define an open fundamental chamber $C^{(i)}$ and a Tits cone $X^{(i)}$ for $W^{(i)}$ in $\a$ by
	$$C^{(i)}:=\{h'\in\a: \langle \alpha_k,h' \rangle>0\,\, \forall k\in I\setminus\{i\}\} \qquad \text{and} \qquad X^{(i)}:=\bigcup_{w\in W^{(i)}} w(\overline{C^{(i)}}).$$
The closure of $C^{(i)}$ is a strict fundamental domain for the $W^{(i)}$-action on $X^{(i)}\subset\a$, so every face of $P^{(i)}(t)\subset X^{(i)}$ is $W^{(i)}$-equivalent to a face that intersects the closed fundamental chamber
	$$\overline{C^{(i)}}=\{h'\in\a: \langle \alpha_k,h' \rangle\geq 0\,\, \forall k\in I\setminus\{i\}\}.$$
In particular, every vertex is $W^{(i)}$-equivalent to a vertex in that set. More precisely, the next proposition asserts that there are only finitely many such vertices and all of them are already contained in the open fundamental chamber $C^{(i)}$. To distinguish between the different notions of chambers, we refer to the $W$-translates of $\overline{C}$ as $W$-chambers and to the $W^{(i)}$-translates of $\overline{C^{(i)}}$ as $W^{(i)}$-chambers and likewise for the open chambers.

\begin{prop}
	For every $t\in I_h$, the set $P^{(i)}(t)$ is the convex hull of a finite union of $W^{(i)}$-orbits. More precisely, there exist finitely many vertices $\{x_1,\ldots,x_m\}$ of $P^{(i)}(t)$ that are contained in $C^{(i)}$ and
	$$P^{(i)}(t)=\conv\bigg( \bigcup_{r=1}^m W^{(i)}\cdot x_r\bigg).$$
\end{prop}

\begin{proof}
	Applying Lemma \ref{lem:face_hull_lemma} to the trivial face $F=P^{(i)}(t)$ yields that $P^{(i)}(t)$ is the convex hull of its vertices. Every such vertex is $W^{(i)}$-equivalent to a vertex in $\overline{C^{(i)}}$, so $P^{(i)}(t)$ is the convex hull of the union of all $W^{(i)}$-orbits of its vertices in $\overline{C^{(i)}}$. We first show that there are only finitely many such vertices. Recall that every vertex is of one of the following two types: It is either a vertex of $\Xi_h$ contained in $H^{(i)}(t)$, i.e. equal to $w(h)$ for some $w\in W$, or it lies in the interior of a unique edge of $\Xi_h$. \\
	First, there can only be finitely many vertices of the first type. If $t\neq\langle \omega_i, w(h) \rangle$ for all $w\in W$, there cannot be any such vertices at all, so assume that $t=\langle \omega_i, w(h) \rangle$ for some $w\in W$ and let $\alpha\in Q_+$ be such that $w^{-1}(\omega_i)=\omega_i-\alpha$. For any other $w'\in W$ with $w'^{-1}(\omega_i)=\omega_i-\beta$, we have $\langle \omega_i,w'(h) \rangle=t$ if and only if $\langle\alpha,h\rangle=\langle \beta,h\rangle$. If $\alpha=\beta$, then $w^{-1}(\omega_i)=w'^{-1}(\omega_i)$, so that $w'w^{-1}\in W^{(i)}$ and $w(h)$ is $W^{(i)}$-equivalent to $w'(h)$. Otherwise, since $\langle \beta,h\rangle>0$ for all $\beta\in Q_+$, the equality $\langle\alpha,h\rangle=\langle \beta,h\rangle$ is only possible for finitely many $\beta\in Q_+$.  \\
	Second, this argument also shows that there can only be finitely many vertices of the second type. Let $x$ be a vertex of $P^{(i)}(t)$ that is contained in the interior of an edge of $\Xi_h$. The endpoints of that edge are $w(h)$ and $ws_k(h)$ for some $w\in W$ and $k\in I$ with $\langle \omega_i,w(h)\rangle>t$ and $\langle\omega_i,ws_k(h)\rangle<t$. By what we have just shown, there are only finitely many $W^{(i)}$-inequivalent elements $w\in W$ with $\langle \omega_i,w(h)\rangle>t$ and from every such vertex $w(h)$, there are only $n$ edges emanating from it, so there are only finitely many $W^{(i)}$-inequivalent edges of $\Xi_h$ that cross $H^{(i)}(t)$. \\
	Hence, there exist finitely many vertices $\{x_1,\ldots,x_m\}$ of $P^{(i)}(t)$ that lie in $\overline{C^{(i)}}$. We now show that they are actually contained in $C^{(i)}$. To this end, it suffices to show that none of these vertices is contained in a wall of $\overline{C^{(i)}}$. If $x_r$ is of the first type, then $x_r$ is $W$-equivalent to $h$ which is assumed to be contained in the open $W$-chamber $C$. This is a subset of $C^{(i)}$, so $x_r$ is not contained in any wall of $\overline{C^{(i)}}$. If $x_r$ is of the second type, it is contained in the interior of a unique edge $L$ of $\Xi_h$ whose endpoints are $w_1(h)$ and $w_2(h)$ for some $w_1,w_2\in W$. By Corollary \ref{cor:graphAdjacency}, this edge crosses precisely one wall, so $\langle \alpha_k,x_r\rangle=0$ for some $k\neq i$ is only possible if $w_2=s_kw_1$. But then $\omega_i$ takes the same value at $w_1(h)$ and $w_2(h)$ which implies that $\omega_i$ is constant on $L$. Since $x_r\in L\cap P^{(i)}(t)$, it follows that $L\subset P^{(i)}(t)$ is an edge of $P^{(i)}(t)$. However, $x_r$ is an interior point of $L$ and a vertex of $P^{(i)}(t)$, a contradiction.
\end{proof}

\noindent In particular, it follows that every vertex of $P^{(i)}(t)$ lies in an open $W^{(i)}$-chamber $w(C^{(i)})$ for some $w\in W^{(i)}$. Moreover, every face of $P^{(i)}(t)$ contains a vertex, so we also conclude that every face of $P^{(i)}(t)$ is $W^{(i)}$-equivalent to a face which contains at least one of the vertices $\{x_1,\ldots,x_m\}$. We denote by $P^{(i)}_{\op{ess}}(t)$ the union of all faces of $P^{(i)}(t)$ which are entirely contained in $C^{(i)}$. Equivalently, since $C^{(i)}$ is convex, a face lies in $P^{(i)}_{\op{ess}}(t)$ if and only if all its vertices are contained in $\{x_1,\ldots,x_m\}$. We call $P^{(i)}_{\op{ess}}(t)$ the \textit{essential part} of $P^{(i)}(t)$. This terminology is motivated by Proposition \ref{prop:essentialHull} below. To prepare for the proof, we need some more information on the faces of $P^{(i)}(t)$. Observe that if $m=1$, then $P^{(i)}(t)$ is the convex hull of a single $W^{(i)}$-orbit and the arguments of this section prove that its face poset is isomorphic to $\Sigma^\ast(W^{(i)},S^{(i)})$. If $m>1$, this is no longer true. Since $W^{(i)}$ acts simply transitively on $W^{(i)}$-chambers, a face in the essential part cannot be stabilized by any non-trivial element of $W^{(i)}$. Conversely, the following result shows that if a face of $P^{(i)}(t)$ is not contained in the essential part, it must have a certain degree of symmetry.

\begin{prop}\label{prop:invFaces}
	Every face of $P^{(i)}(t)$ is invariant under a parabolic subgroup of $W^{(i)}$. More precisely, for every face $F$ of $P^{(i)}(t)$, there exist $1\leq t\leq m$, $r_1,\ldots,r_t\in\{1,\ldots,m\}$, $w\in W^{(i)}$ and a standard subgroup $W_J\subset W^{(i)}$ such that
	$$F=\op{conv}\bigg(\bigcup_{s=1}^t wW_J\cdot x_{r_s}\bigg).$$
\end{prop}

\noindent The second statement is indeed a refinement of the first because if $F$ has the asserted shape, then $F$ is invariant under the parabolic subgroup $wW_Jw^{-1}\subset W^{(i)}$.

\begin{proof}
	Since $F$ is the convex hull of its vertices, it suffices to show that the vertices of $F$ form a union of orbits of a standard coset of $(W^{(i)},S^{(i)})$. After replacing $F$ by $w^{-1}(F)$ for a suitable $w\in W^{(i)}$, we may assume that at least one vertex of $F$ is contained in $C^{(i)}$. We are going to show that there exists a unique maximal subset $J\subset I\setminus\{i\}$ such that $F$ is invariant under the standard subgroup $W_J\subset W^{(i)}$, so that $w(F)$ will be invariant under the parabolic subgroup $wW_Jw^{-1}$. If all vertices of $F$ are already contained in the set $\{x_1,\ldots,x_m\}$, the claim is satisfied only by $J=\emptyset$. Otherwise, there exists an edge $L\subset F$ which connects a vertex $x_r\in C^{(i)}$ to a vertex $y\in w(C^{(i)})$ for some $w\in W^{(i)}\setminus\{1\}$. We claim that $w=s_k$ for some $k\neq i$ and $y=s_k(x_r)$. \\
	Let $w=s_{i_1}\ldots s_{i_l}$ be a reduced decomposition of $w$ and set $k:=i_1\in I\setminus\{i\}$. Then we have $\langle \alpha_k,x_r\rangle>0$ as $x_r\in C^{(i)}$ and $\langle \alpha_k,y\rangle=\langle w^{-1}(\alpha_k),w^{-1}(y)\rangle<0$ because $w^{-1}(\alpha_k)\in\Phi^{(i)}_-$ by \eqref{PositivityCriterion} and $w^{-1}(y)\in C^{(i)}$. Hence, $x_r$ and $y$ lie in open $W^{(i)}$-chambers on different sides of the hyperplane $V_{\alpha_k}:=\op{ker}(\alpha_k)$, so the edge $L$ crosses this plane at an interior point. Since the faces of $P^{(i)}(t)$ are $W^{(i)}$-invariant, $L':=s_k(L)$ is also an edge of $P^{(i)}(t)$. As $s_k$ fixes an interior point of $L$, the two edges $L'$ and $L$ intersect at a point that lies in the relative interior of both of them. This is only possible if $L'=L$, which forces $w=s_k$ and $y=s_k(x_r)$. In particular, it shows that $L$ is invariant under the rank-one standard subgroup of $W^{(i)}$ generated by $s_k$. \\
	Next, we show that the entire face $F$ is $s_k$-invariant. To this end, let $E\subset\a$ be a supporting hyperplane for $F$ and denote by $E':=E-x_r\subset\a$ the corresponding parallel translate containing the origin, i.e. the linear subspace corresponding to $E$. Observe that $E$ contains the infinite line $L_0$ through the two distinct points $x_r$ and $y=s_k(x_r)$ which crosses the wall $V_{\alpha_k}$, so $E'$ contains the one-dimensional $s_k$-invariant subspace $L_0':=L_0-x_r$ which is not contained in $V_{\alpha_k}$. Therefore, $E'$ and $V_{\alpha_k}$ intersect in a $(d-2)$-dimensional subspace and every point $v\in E'$ can be uniquely written as $v=v_1+v_2$ with $v_1\in L_0'$ and $v_2\in E'\cap V_{\alpha_k}$. This yields that $s_k(v)=s_k(v_1)+s_k(v_2)=s_k(v_1)+v_2\in E'$, so $E'$, and hence $F$, is $s_k$-invariant. Continuing in this way for all $W^{(i)}$-chambers adjacent to $C^{(i)}$ that $F$ meets, we conclude that $F$ is invariant under a unique maximal standard subgroup $W_J$ of $W^{(i)}$. \\
	Hence, if $\{x_{r_1},\ldots,x_{r_t}\}$ are the vertices of $F$ contained in $C^{(i)}$, then the set $W_J\cdot \{x_{r_1},\ldots,x_{r_t}\}$ also consists of vertices of $F$. We claim that these are all the vertices of $F$. Assume that this is not the case, then there exists a vertex $y_1$ in a chamber $w_1(C^{(i)})$ for some $w_1\in W^{(i)}\setminus W_J$. Since the underlying graph of $P^{(i)}(t)$ is connected, we may choose $y_1$ in such a way that it is connected by an edge $L$ to a vertex of the form $y_2=w_2(x_{r_s})$, where $s\in\{1,\ldots,t\}$ and $w_2\in W_J$. As $F$ is $W_J$-invariant, $w_2^{-1}(L)$ is an edge of $F$ connecting $x_{r_s}$ to a vertex $y':=w_2^{-1}(y_1)\in w_2^{-1}w_1(C^{(i)})$. Exactly as in the first part of the proof, we conclude that $w_2^{-1}w_1=s_l$ for some $l\neq i$, but $s_l\notin W_J$. This contradicts the maximality of $W_J$ and shows that the set of vertices of $F$ is the union of the $W_J$-orbits of its vertices in $C^{(i)}$. The proof is therefore complete.
\end{proof}

\noindent In order to construct the set $P^{(i)}(t)$ from its vertices, one first has to form the union of a finite number of $W^{(i)}$-orbits and then take the convex hull of the resulting set. However, with the previous proposition at hand, we can now show that the essential part allows us to exchange these operations and instead express $P^{(i)}(t)$ as the union of convex hulls of $W^{(i)}$-orbits. Observe that if $m>1$, then the essential part contains uncountably many points.

\begin{prop} \label{prop:essentialHull}
	Every point in $P^{(i)}(t)$ lies in the convex hull of the $W^{(i)}$-orbit of a point in $P^{(i)}_{\op{ess}}(t)$, i.e.
	$$P^{(i)}(t)=\bigcup_{y\in P^{(i)}_{\op{ess}}(t)} \conv(W^{(i)}\cdot y).$$
\end{prop}

\begin{proof}
	Let $x\in P^{(i)}(t)$ be arbitrary and let $F$ be the smallest face of $P^{(i)}(t)$ containing $x$. We prove the claim by induction on the dimension of $F$. Since $P^{(i)}(t)$ is $W^{(i)}$-invariant, we may w.l.o.g. assume that $x\in \overline{C^{(i)}}$, which implies that one of the vertices of $F$ is contained in $C^{(i)}$. Let $x_r$ for some $r\in\{1,\ldots,m\}$ be one such vertex. \\
	If $\op{dim}(F)=0$, then $x=x_r\in P^{(i)}_{\op{ess}}(t)$ and there is nothing to show. Assume now that $\op{dim}(F)>0$. If all vertices of $F$ are in $\{x_1,\ldots,x_m\}$, then $F\subset P^{(i)}_{\op{ess}}(t)$ and we are done. Otherwise, Proposition \ref{prop:invFaces} shows that $F$ is invariant under a non-trivial standard subgroup $W_J\subset W^{(i)}$. Assume first that $x$ is not a global fixed-point of $W_J$ and choose some $k\in J$ such that $s_k(x)\neq x$. Then $\omega_k$ is not bounded from above on the infinite line $L_0$ spanned by $x$ and $s_k(x)$, so $L_0$ intersects the boundary of $F$ in at least one point $x_+$ which is contained in a proper face of $F$. By the induction hypothesis, there exists a point $y\in P^{(i)}_{\op{ess}}(t)$ such that $x_+\in\op{conv}(W^{(i)}\cdot y)$. Since $x$ and $s_k(x)$ are contained in the relative interior of $F$ and $x_+$ in the boundary, $x_+$ cannot lie between $x$ and $s_k(x)$ on $L_0$. This implies that $x$ is a convex combination of $x_+$ and $s_k(x_+)$ and we conclude that $x\in\op{conv}(W^{(i)}\cdot y)$. \\
	If $s_k(x)=x$ for all $k\in J$, then the line $L_0$ is not well-defined, but this issue is easily dealt with. By minimality of $F$, $x$ has to lie in the relative interior of $F$. Let $E$ be a supporting hyperplane for $F$ and consider its parallel translate $E':=E-x$ containing the origin. Let $k\in J$ be arbitrary, then $s_k\in W_J$ acts on $\a$ as a non-trivial involution, so one of its eigenvalues is $-1$. Moreover, it preserves $E'$ and acts non-trivially on it. Hence, there exists a non-zero vector $v\in E'$ such that $s_k(v)=-v$. Choosing $v\neq 0$ sufficiently small, we have $x':=x+v\in F$ and $x$ can be written as the non-trivial convex combination
	$$x=\frac{1}{2}(x+v+x-v)=\frac{1}{2}(x'+s_k(x')).$$
	The same argument as above applied to the line spanned by $x'$ and $s_k(x')$ yields that $x'$, and hence $x$, lies in the convex hull of $W^{(i)}\cdot y$ for some $y\in P^{(i)}_{\op{ess}}(t)$.
\end{proof}

\begin{exam}
	We illustrate the content of Proposition \ref{prop:essentialHull} in Figure \ref{fig:EssentialParts}. Both pictures show the union of the convex hulls of three $W^{(i)}$-orbits, where $W^{(i)}$ is the symmetric group $S_3$ in (a) and $W^{(i)}=(\Z_2)^3$ in (b). The set $P^{(i)}(t)$ is the convex hull of these three orbits, which is not depicted in the figure, and the essential part is highlighted in red. In both examples, it is easy to imagine that as a point moves over the red region, the corresponding $W^{(i)}$-orbits trace out the entire set $P^{(i)}(t)$. In Figure (a), we have also drawn the reflecting hyperplanes which bound the Weyl chambers. In Figure (b), these would be the coordinate hyperplanes and the closed fundamental chamber is the octant of $\R^3$ where all coordinates are non-negative.
	\begin{figure}[h] 
		\centering
		\begin{subfigure}[t]{0.3\linewidth}
			\includegraphics[width=\linewidth]{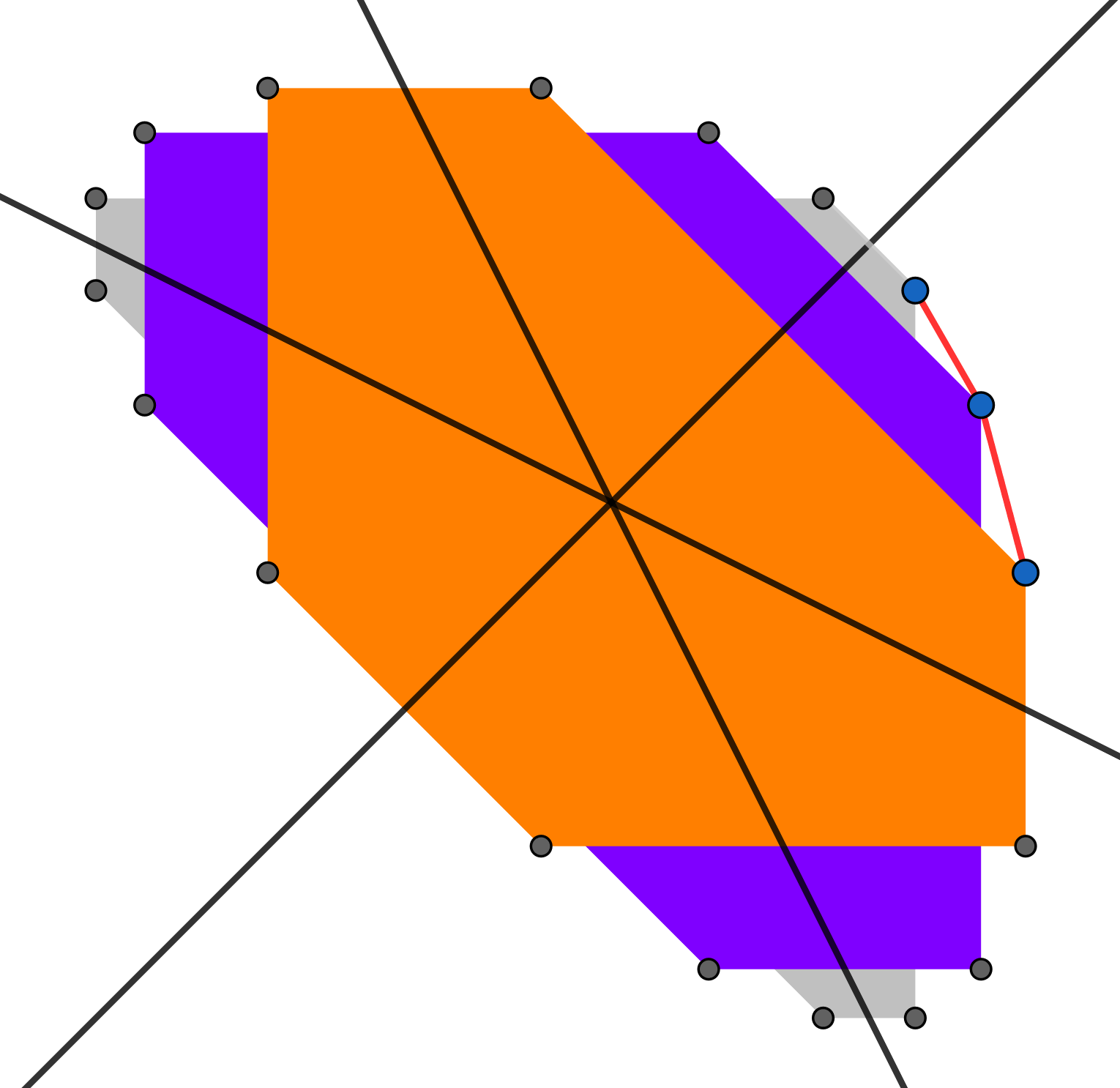}
			\caption{$W^{(i)}=S_3$}
		\end{subfigure}
		\hspace{2cm}
		\begin{subfigure}[t]{0.3\linewidth}
			\includegraphics[width=\linewidth]{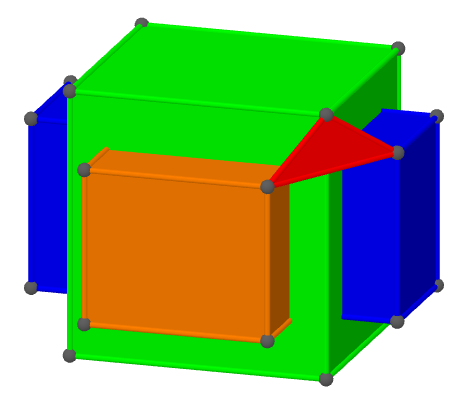}
			\caption{$W^{(i)}=(\Z_2)^3$}
		\end{subfigure}	 
	\caption{Two configurations of $P^{(i)}_{\op{ess}}(t)$.}
	\label{fig:EssentialParts}
	\end{figure}
\end{exam}

\noindent Proposition \ref{prop:essentialHull} will play an important role in the proof of Theorem \ref{thm:KMKostant}. We close this section by briefly outlining how the correspondence between minimal edge paths in $\Xi_h$ and minimal galleries in the Coxeter complex $\Sigma(W,S)$ (cf. Corollary \ref{cor:graphAdjacency}) provides a method to determine the number $m=m(t)$ of vertices of $P^{(i)}_{\op{ess}}(t)$ for a varying $t\in I_h$, which is especially interesting in the case of finite index. We use the same notation as in Lemma \ref{lem:IndexDistinction}. Suppose that $[W:W^{(i)}]<\infty$, let $w_+$ be the longest word of $W_{I\setminus I_\infty}$ and let $j\in I$ be the unique index such that $s_j=w_+s_iw_+$. Denote by $w_-$ the unique element of minimal length in the coset $W^{(i)}w_+=w_+W^{(j)}$. The faces $F^{(i)}$ and $w_+(F^{(j)})=\conv(W^{(i)}\cdot w_+(h))$ of $\Xi_h$ are ``opposite'' in the sense that they have the same shape, their defining hyperplanes are parallel and $\Xi_h$ lies between them. The point $w_-(h)$ is the unique vertex of $w_+(F^{(j)})$ whose distance to $h$ in the underlying graph of $\Xi_h$ is minimal. As this graph is isomorphic to the underlying chamber system of $\Sigma(W,S)$, minimal edge paths from $h$ to $w_-(h)$ are in bijective correspondence with reduced decompositions of $w_-$. In addition, every such path is contained in the open $W^{(i)}$-chamber $C^{(i)}$ and every vertex in $C^{(i)}$ lies on such a path. With this information, it is not difficult to show that every minimal edge path from $h$ to $w_-(h)$ intersects $H^{(i)}(t)$ in one of the vertices $\{x_1,\ldots,x_m\}$ of $P^{(i)}_{\op{ess}}(t)$ and, conversely, every vertex lies on such a minimal path. In particular, if $w_-$ has a unique reduced decomposition, then there exists only one minimal edge path from $h$ to $w_-(h)$ and we have $m(t)=1$ for all $t\in I_h$. In this case, $P^{(i)}(t)$ is the convex hull of a single $W^{(i)}$-orbit. This happens, for example, if $W$ is the symmetric group $S_n$ and $i\in\{1,n-1\}$, but it is not the case if $n\geq 4$ and $i\in\{2,\ldots,n-2\}$. We point out explicitly, however, that $m(t)$ is usually not equal to the number of reduced decompositions of $w_-$ because distinct minimal edge paths from $h$ to $w_-(h)$ may well intersect $H^{(i)}(t)$ in the same point. For example, every such path begins with the edge $[h,s_i(h)]$ and ends with $[w_-s_j(h),w_-(h)]$, so that $m(t)=1$ for all $t\in [-\langle\omega_j,h\rangle,-\langle\omega_j,s_j(h)\rangle]\cup [\langle\omega_i,s_i(h)\rangle,\langle\omega_i,h\rangle]$. In between these values, minimal edge paths typically split and merge. Since the vertices of $P^{(i)}(t)$ are simply the intersection points of $H^{(i)}(t)$ with edges of $\Xi_h$, varying $t$ can only change $m(t)$ if this causes the hyperplane $H^{(i)}(t)$ to begin or to stop intersecting an edge of $\Xi_h$. Hence, $m(t)$ is constant on the connected components of the complement of the discrete set $\{\langle\omega_i,w(h)\rangle\mid w\in W\}$ in $I_h$. \\
If $[W:W^{(i)}]=\infty$, the element $w_-\in W$ is not defined which mirrors the fact that the length of a left $W^{(i)}$-reduced element in $W$ is unbounded. In this case, there is no canonical vertex of $\Xi_h$ that all minimal edge paths in $C^{(i)}$ starting at $h$ can be extended to. Nevertheless, the set of all $w\in W$ such that $w(h)$ is the first vertex in a minimal edge path in $C^{(i)}$ starting at $h$ with $\langle\omega_i,w(h)\rangle\leq t$ is still finite and $m(t)$ can be computed as in the finite-index case by counting reduced decompositions of these elements and investigating whether the associated minimal paths intersect $H^{(i)}(t)$ in distinct points.

%% file: Sections/KMKostantPaper_KMGroups.tex
\section{Kac-Moody groups} \label{KMGroups}

Having only dealt with Kac-Moody algebras and their Weyl groups so far, we now introduce the class of groups that will be relevant to our discussion. Kac-Moody groups can be defined in many different ways and varying levels of generality. A common approach that was introduced by J. Tits, following the scheme-theoretic construction of the classical Chevalley groups, is to define a Kac-Moody group over a field $\K$ as the group $\mc{G}(\K)$ of $\K$-points of a group functor $\mc{G}:\text{$\Z$-alg}\to\text{Grp}$ satisfying a list of axioms (cf. \cite[Section 2]{Tits}). However, Tits also shows that over fields, this functor is uniquely determined (up to a certain non-degeneracy assumption) and the group $\mc{G}(\K)$ can be given an explicit presentation by generators and relations (cf. \cite[Section 3.6 and 3.9, Theorem 1]{Tits}). For our purposes, it will suffice to consider Kac-Moody groups defined via this presentation which is called the \textit{constructive Tits functor} in \cite[Section 8.3]{Remy}. The set of imaginary roots does not occur in this definition of Kac-Moody groups, so we continue to denote the set of real roots by $\Phi$ instead of $\Phi^{\op{re}}$. In the first part of this chapter, however, we do not require Assumption \ref{assum:GlobalAssumption}. For any field $\K$, we write $\K^\ast:=\K\setminus\{0\}$ for its group of invertible elements. 
\medskip \newline
Let $\mc{D}=(I,\A,\Lambda,(c_i)_{i\in I},(h_i)_{i\in I})$ be an arbitrary Kac-Moody root datum of rank $d$, not necessarily free or cofree here. For any field $\K$, we set $T(\K):=\text{Hom}_{\text{Grp}}(\Lambda,\K^\ast)\cong\Lambda^\vee\otimes_\Z \K^\ast$. Hence, $T(\K)$ is a free abelian group of rank $d$ isomorphic to $(\K^\ast)^d$. For $r\in \K^\ast$ and $h\in\Lambda^\vee$, we denote by $r^h\in T(\K)$ the homomorphism $\lambda\mapsto r^h(\lambda):=r^{\langle \lambda,h\rangle}$. If $\{v_1,\ldots,v_d\}$ is a $\Z$-basis of $\Lambda^\vee$, then $T(\K)$ is generated by $\{r^{v_i}\mid 1\leq i\leq d,\,r\in\K^\ast\}$. The canonical action of the Weyl group on $\Lambda$ induces an action of $W$ on $T(\K)$ by $w(t):=t\circ w^{-1}$ for $w\in W$ and $t\in T(\K)$. If $t=r^h$ for some $r\in\K^\ast$ and $h\in\Lambda^\vee$, then $w(r^h)=r^{w(h)}$. \\
For any Kac-Moody root datum $\mc{D}$, the \textit{(split, minimal) Kac-Moody group of type $\mc{D}$ over $\K$} is the group $G(\K)$ generated by $T(\K)$ and a family of subgroups $(U_\alpha(\K))_{\alpha\in\Phi}$ indexed by the real roots, subject to the relations in \cite[Section 3.6]{Tits}.\footnote{We do not give the exact presentation here since its formulation requires concepts that we have not discussed. Also, their precise form is largely irrelevant to our discussion. What matters are the algebraic and representation-theoretic properties listed below that follow from the relations.} Here, $U_\alpha(\K)$ is an abelian subgroup of $G(\K)$ isomorphic to the additive group $(\K,+)$. We choose an isomorphism $x_\alpha:\K\to U_\alpha(\K)$ and write $U_\alpha(\K)=\{x_\alpha(r) \mid r\in\K\}$. We call $T(\K)$ the \textit{(standard $\K$-split) torus} and $(U_\alpha(\K))_{\alpha\in\Phi}$ the \textit{root groups} of $G(\K)$. Every root group is normalized by the torus. In fact, for every $\alpha\in\Phi,\,r\in\K$ and $t\in T(\K)$, we have $tx_\alpha(r)t^{-1}=x_\alpha(t(c(\alpha))r)$. We point out that we have suppressed the root datum $\mc{D}$ from the notation, but the torus $T(\K)$ clearly depends on $\mc{D}$ and so do some of the defining relations that we have not mentioned.

\begin{rmk} \label{rmk:KMBuilding}
	Although we will not use this theory explicitly, we mention for completeness that the tuple $(G(\K),(U_\alpha(\K))_{\alpha\in\Phi},T(\K))$ is an \textit{RGD-system of type $(W,S)$} (see \cite[Section 8.6]{AB}). This is a simple consequence of the fact that the real roots $\Phi$ are in bijective correspondence with the roots of the Coxeter complex $\Sigma(W,S)$ and the commutator relations between the root groups that are built into the presentation of $G(\K)$ (cf. \cite[Theorem 7.69]{Marquis}). In particular, $G(\K)$ acts strongly transitively on an associated Moufang twin building which is a valuable tool to investigate its algebraic structure. We refer to \cite[Chapter 8]{AB} for more information. 
\end{rmk}

\noindent To simplify the notation in the following discussion, we fix a field $\K$ and write $G,T,U_\alpha,\ldots$ instead of $G(\K),T(\K),U_\alpha(\K),\ldots$ and only indicate the field if there is danger of confusion. For every $i\in I$, set $\tilde{s}_i:=x_{\alpha_i}(1)x_{-\alpha_i}(1)x_{\alpha_i}(1)$ and let $\tilde{W}$ be the subgroup of $G$ generated by $\{\tilde{s}_i \mid i\in I\}$. It is called the \textit{extended Weyl group}. It induces the natural $W$-action on $T$ in the sense that for every $w\in W$ there exists some $\tilde{w}\in\tilde{W}$ such that $\tilde{w}t\tilde{w}^{-1}=w(t)$ for all $t\in T$. Moreover, it also permutes the root groups by the rule $\tilde{w}x_\alpha(r)\tilde{w}^{-1}=x_{w(\alpha)}(\pm r)$ for all $\alpha\in\Phi$ and $r\in\K$, where the ambiguous sign need not concern us here. \\
For any subset $J\subset I$, we set $\Phi(J):=\Phi\cap\bigoplus_{j\in J}\Z\alpha_j$, $\Phi_+(J):=\Phi_+\cap\Phi(J)$ and $\Phi^\perp_+(J):=\Phi_+\setminus \Phi_+(J)$ and define the following three subgroups of $G$:
	\begin{align*}
		G_J &:= \langle T, U_\alpha \mid \alpha\in\Phi(J) \rangle \\
		U_+^\perp(J) &:= \langle U_\alpha \mid \alpha\in\Phi_+^\perp(J) \rangle \\			
		P_J &:= \langle T, U_\alpha \mid \alpha\in\Phi_+ \cup \Phi(J) \rangle
	\end{align*}
The group $P_J$ is called the \textit{standard parabolic subgroup of type $J$} and it has a Levi decomposition
	\begin{equation}
		P_J\cong G_J \ltimes U_+^\perp(J). \label{LeviDecomp} 
	\end{equation}
The Levi factor $G_J$ is again a Kac-Moody group associated to the Kac-Moody root datum $\mc{D}(J):=(J,\A_J,\Lambda,(c_j)_{j\in J},(h_j)_{j\in J})$, where $\A_J$ is the submatrix of $\A$ with rows and columns indexed by $J$. Note that if $\mc{D}$ is free (resp. cofree), then $\mc{D}(J)$ has the same property. If $J=\emptyset$, we set $B_+:=P_\emptyset$ and $U_+:=U_+^\perp(\emptyset)$ and obtain $B_+\cong T\ltimes U_+$. The group $B_+$ is called the \textit{positive Borel subgroup} of $G$ and $U_+$ is the subgroup of $G$ generated by all $U_\alpha$ with $\alpha\in\Phi_+$.
\medskip \newline
Next, we transfer some concepts from the representation theory of Kac-Moody algebras (cf. Chapter \ref{KMAlgebras}) to the group level. Let $\g$ be the Kac-Moody algebra of type $\mc{D}$ over $\K$ with Cartan subalgebra $\h\subset\g$ and let $\lambda\in\h^\ast$ be a dominant integral weight. Since the highest-weight module $\rho_\lambda:\g\to\mf{gl}(L(\lambda))$ is integrable, it can be lifted to a $\K$-linear representation $\hat{\rho}_\lambda:G\to \GL(L(\lambda))$ by
\begin{equation}
	\begin{aligned}
		\hat{\rho}_\lambda(x_\alpha(r)) &:= \exp(\rho_\lambda(re_\alpha)), &\qquad &\alpha\in\Phi,\,r\in\K, \\
		\hat{\rho}_\lambda(t)v_\mu &:= t(\mu)v_\mu, &\qquad &t\in T,\,\mu\in \op{wt}(L(\lambda)),\, v_\mu\in L(\lambda)_\mu,
	\end{aligned} \label{GroupRep}
\end{equation}
where $e_\alpha\in\g_\alpha$ is a suitably chosen non-zero vector (cf. \cite[Remark 7.6]{Marquis}). In particular, it follows from \eqref{WeightSpaceTranslation} and \eqref{GroupRep} that for $\alpha\in\Phi$, $r\in\K$ and $v_\mu\in L(\lambda)_\mu$, we have $\rho_\lambda(re_\alpha)^{N+1}v_\mu=0$ for some $N\in\N_0$ and
	\begin{equation}
		\hat{\rho}_\lambda(x_\alpha(r))v_\mu=v_\mu+\sum_{k=1}^N \text{(an element of $L(\lambda)_{\mu+kc(\alpha)}$).} \label{RootGroupAction}
	\end{equation}
Let us now specialize to the case that the field $\K$ is either $\R$ or $\C$. Let $\{v_1,\ldots,v_d\}$ be a $\Z$-basis of $\Lambda^\vee$, denote by $\a$ the real Cartan subalgebra of $\g_\R$ and by $\h$ the complex Cartan subalgebra of $\g_\C$. There exists a well-defined surjective \textit{exponential map}
	$$\exp:\h\to T(\C), \qquad \sum_{i=1}^d r_iv_i \mapsto \prod_{i=1}^d (e^{r_i})^{v_i}, \qquad r_1,\ldots,r_d\in\C,$$
which corresponds to the coordinate-wise exponential map $\C\to \C^\ast$ under the isomorphisms $\h\cong\C^d$ and $T(\C)\cong (\C^\ast)^d$. It restricts to an injective map $\exp:\a\to T(\R)$ whose image is isomorphic to $(\R_{>0})^d$ and will be denoted by $A$. It can be viewed as a subgroup of $G(\R)$ and of $G(\C)$. The polar decomposition in $\C$ and its trivial restriction to $\R$ imply that $T(\K)\cong M\times A$, where $M$ is isomorphic to $(S^1)^d$ if $\K=\C$ and to $(\Z_2)^d$ if $\K=\R$. The extended Weyl group $\tilde{W}\subset G$ normalizes $A$ and $\exp:\a\to A$ is $W$-equivariant. \\
Moreover, the compact involution $\theta:\g\to\g$ possesses an analogue on the group level. More precisely, there exists a unique involutive automorphism $\Theta:G\to G$ which satisfies $\Theta(x_\alpha(r))=x_{-\alpha}(\overline{r})$ for all $\alpha\in\Phi$ and $r\in\K$ and $\Theta(r^{v_i})=\overline{r}^{-v_i}$ for all $i\in\{1,\ldots,d\}$ and $r\in\K^\ast$. Here, the bars denote complex conjugation which can be omitted if $\K=\R$. It is called the \textit{compact involution} of $G$. Its fixed-point set is denoted by $K$ and called the \textit{compact form} of $G$. By definition of the extended Weyl group, $\tilde{W}$ is contained in $K$. \\
Using the fact that all panels in the twin building associated to $G$ (cf. Remark \ref{rmk:KMBuilding}) are isomorphic to the projective line over $\K$, it is not difficult to show that the action of $K$ on this building is chamber-transitive and hence $G=KB_+$. Moreover, we have $B_+=T\ltimes U_+$ and one easily verifies that $K\cap U_+=\{1\}$ and $K\cap T=M$. Since $T\cong M\times A$, this shows that $G$ has an \textit{Iwasawa decomposition} $G=KAU_+$ in the sense that
\begin{equation}
	\text{the multiplication map $K\times A\times U_+\to G$ is bijective.}
\end{equation}
Such a decomposition for certain types of Kac-Moody groups has been established, for example, in \cite[Proposition 5.1]{KP1} or \cite[Theorem 3.31]{FHHK}. The argument that we have indicated is inspired by \cite[Section 3]{DHK}. Note that the group $A$ is the same in the real and complex case, only $K$ and $U_+$ depend on the field. The corresponding projection onto the $A$-component will be denoted by
	$$\Pi:G\to A, \qquad g=kau\mapsto a.$$
We study a particular property of this map in our main theorem. A key ingredient of the proof will be the following facts about highest-weight representations of $G$. As discussed at the end of Chapter \ref{KMAlgebras}, if $\lambda\in\a^\ast\subset\h^\ast$ is a dominant integral weight, there exists a contravariant symmetric bilinear or Hermitian form $H$ on the irreducible real or complex highest-weight module $L(\lambda)$. Using that $G$ is generated by the torus $T$ and all root groups $U_\alpha$ with $\alpha\in\Phi$, it is straightforward to deduce from \eqref{GroupRep} and \eqref{ContravariantForm} that the corresponding representation of $G$ on $L(\lambda)$ satisfies 
	\begin{equation}
		H(g\cdot v,w)=H(v,\Theta(g^{-1})\cdot w) \qquad \text{for all $g\in G$ and $v,w\in L(\lambda)$.} \label{ContravariantForm2}
	\end{equation}
Therefore, $H$ is $K$-invariant. If the GCM $\A$ is symmetrizable, then $H$ is positive definite, so that $K$ acts on $L(\lambda)$ by unitary operators. Moreover, $A$ acts diagonally on $L(\lambda)$ with positive eigenvalues. In fact, if $a=\exp(h)$ with $h=\sum_{i=1}^d r_iv_i\in\a$, then a simple computation using \eqref{GroupRep} shows that $a(\mu)=e^{\langle\mu,h\rangle}$ for every weight $\mu$ of $L(\lambda)$, which is positive since $\langle\mu,h\rangle\in\R$. The action of $U_+$ is determined by \eqref{RootGroupAction}. In particular, if Assumption \ref{assum:GlobalAssumption} is satisfied, this discussion applies to the functionals in the dual basis $\mc{B}^\ast=\{\omega_1,\ldots,\omega_n,\phi_{n+1},\ldots,\phi_d\}\subset\a^\ast$ of Proposition \ref{prop:AssumptionConsequences}.

%% file: Sections/KMKostantPaper_MainTheorem.tex
\section{The convexity theorem} \label{MainTheorem}

We now come to our main theorem. Throughout this chapter, we suppose that Assumption \ref{assum:GlobalAssumption} is in effect. Let $G$ be the split Kac-Moody group over $\K\in\{\R,\C\}$ associated to a free and cofree Kac-Moody root datum $\mc{D}$ of rank $d\geq 2n-\op{rank}(\A)$ and denote by $\a$ the Cartan subalgebra of $\g_\R$ with basis $\mc{B}=\{\alpha_1^\vee,\ldots,\alpha_n^\vee,h_{n+1},\ldots,h_d\}$. Let $G=KAU_+$ be an Iwasawa decomposition of $G$ and define
	$$\Pi:G\to A, \qquad g=kau\mapsto a$$
as the corresponding projection onto the $A$-component. It is clear that $\Pi(kg)=\Pi(g)$ for all $k\in K$ and $g\in G$, but there is no obvious way how to compute $\Pi(gk)$. The following theorem will answer this question in the special case that $g=a\in A$ is contained in the exponential of an open chamber of the Tits cone $X\subset\a$. Recall that the exponential map restricts to a $W$-equivariant isomorphism $\a\to A$ whose inverse will be denoted by $\log:A\to\a$.
	
\begin{thm} \label{thm:KMKostant}
	If $a\in A$ is contained in the exponential image of an open chamber of the Tits cone $X\subset\a$, then
	$$\{\Pi(ak)\mid k\in K\}=\exp(\conv(W\cdot \log(a))).$$
\end{thm}

\noindent In order to prove Theorem \ref{thm:KMKostant}, let us fix some notation. Assume that $a\in A$ is fixed and set $h:=\log(a)\in\a$. By assumption, $h$ is contained in an open chamber of $X$. We write $\Xi_h:=\conv(W\cdot h)$ as in Chapter \ref{ConvexHulls} and define
$$\Pi_a:=\{\Pi(ak)\mid k\in K\}=\Pi(aK).$$
Thus, we need to show that $\Pi_a=\exp(\Xi_h)$. The proof will consist of two steps: First, we show that $\Pi_a\subset\exp(\Xi_h)$ by using tools from the representation theory of $G$. Afterwards, we prove that $\exp(\Xi_h)\subset\Pi_a$ by induction on the size $n$ of the GCM $\A$. Note that both sets do not depend on the $W$-translate of $a$ (resp. $h$). This is obvious for $\Xi_h$ and for $\Pi_a$, we may argue as follows: If $w\in W$ is arbitrary, then we may choose a representative $\tilde{w}\in\tilde{W}\subset K$ such that conjugation by $\tilde{w}$ induces the action of $w$ on $A$ and we have 
$$\Pi_{w(a)}=\Pi(w(a)K)=\Pi(\tilde{w}a\tilde{w}^{-1}K)=\Pi(a\tilde{w}^{-1}K)=\Pi(aK)=\Pi_a.$$
Thus, we may w.l.o.g. assume that $h\in C$ and $a\in\exp(C)$.

\subsection{Proof of the main theorem} \label{MainTheoremProof}

\subsubsection{Step 1: Highest-weight representations} \label{MainTheoremProof_Step1}

We begin by establishing that $\Pi_a\subset\exp(\Xi_h)$. Let $k\in K$ be arbitrary and write $ak\in G$ according to its Iwasawa decomposition as 
$$ak=k'a'u' \qquad \text{with $k'\in K,\,a'\in A,\, u'\in U_+$.}$$
In order to prove that $h':=\log(a')$ is contained in $\Xi_h$, we show that $h'$ satisfies the inequalities 
	$$\begin{cases}
		\langle \omega_i,w(h')\rangle\leq \langle\omega_i,h \rangle & \text{for all $w\in W$ and $i\in I$} \\
		\langle \phi_i,h'\rangle=\langle\phi_i,h \rangle & \text{for all $i\in \{n+1,\ldots,d\}$}
	\end{cases}$$
of Corollary \ref{cor:inequality1} by studying the action of $G=KAU_+$ in suitable highest-weight representations. By Proposition \ref{prop:AssumptionConsequences} and the discussion in Chapter \ref{KMGroups}, the dual basis $\mc{B}^\ast=\{\omega_1,\ldots,\omega_n,\phi_{n+1},\ldots,\phi_d\}$ of $\a^\ast$ consists of dominant integral weights and the group $G$ acts on the corresponding real or complex highest-weight modules $L(\omega_1),\ldots,L(\omega_n), L(\phi_{n+1}),\ldots,L(\phi_d)$. Each of these modules carries a positive definite, contravariant, symmetric bilinear or Hermitian form $H$ such that weight spaces belonging to different weights are orthogonal and by \eqref{ContravariantForm2}, the action of the subgroup $K\subset G$ is unitary. The subgroup $A\subset G$ acts diagonally on these modules with positive real eigenvalues and the action of $U_+$ is determined by \eqref{RootGroupAction}. Moreover, the modules $L(\phi_{n+1}),\ldots,L(\phi_d)$ are one-dimensional and the action of $U_+$ is trivial. We write $\hat{\rho}_i$ for the corresponding homomorphism $G\to\GL(L(\omega_i))$ if $i\in I$ and $G\to \GL(L(\phi_i))$ if $i\in\{n+1,\ldots,d\}$ and denote the associated representations of the Kac-Moody algebra $\g$ by $\rho_1,\ldots,\rho_d$.
\medskip \newline
First, we fix an arbitrary $i\in\{n+1,\ldots,d\}$ and show that $\langle \phi_i,h\rangle=\langle\phi_i,h' \rangle$. Since $L(\phi_i)$ is one-dimensional and the $K$-action is unitary, we have $\hat{\rho}_i(K)\subset\{\pm \id\}$ and thus 
	$$\hat{\rho}_i(ak) v=\pm\hat{\rho}_i(a) v=\pm a(\phi_i)v$$
for any vector $v\in L(\phi_i)$. Moreover, $U_+$ acts trivially on $L(\phi_i)$ which implies that
	$$\hat{\rho}_i(k'a'u') v=\hat{\rho}_i(k'a') v=a'(\phi_i)\hat{\rho}_i(k') v=\pm a'(\phi_i)v.$$
Since $ak=k'a'u'$, these two vectors must coincide. Now $a(\phi_i)$ and $a'(\phi_i)$ are both positive real numbers, so if we choose $v\neq 0$, then equality is only possible if $a(\phi_i)=a'(\phi_i)$ or, equivalently, $e^{\langle\phi_i,h\rangle}=e^{\langle\phi_i,h'\rangle}$. As the exponents are real numbers, this is equivalent to $\langle\phi_i,h\rangle=\langle\phi_i,h'\rangle$.
\medskip \newline
Thus, it only remains to show that if $w\in W$ and $i\in I$ are arbitrary, then $\langle\omega_i,w(h')\rangle\leq\langle\omega_i,h\rangle$. Since the set of weights of $L(\omega_i)$ is $W$-invariant, $\mu:=w^{-1}(\omega_i)$ is a weight of $L(\omega_i)$ and we may choose a vector $v\in L(\omega_i)_\mu$ with $||v||:=H(v,v)=1$. As we assume $\mc{D}$ to be free, every weight of $L(\omega_i)$ is of the form $\omega_i-\alpha$ for some $\alpha\in Q_+$. The eigenvalues of the diagonal map $\hat{\rho}_i(a)$ are precisely the exponentials of the eigenvalues of $\rho_i(h)$. Since $h\in C$, we have $\langle \alpha,h\rangle>0$ for all $\alpha\in Q_+$, so the largest eigenvalue is $a(\omega_i)=e^{\langle \omega_i,h\rangle}$. Moreover, $u'\in U_+$ is a product of elements in positive root groups, so \eqref{RootGroupAction} implies that for every $\nu\in\op{wt}(L(\omega_i))$ with $\mu<\nu\leq\omega_i$, there exists a suitable weight-$\nu$ vector $v_\nu\in L(\omega_i)_\nu$ such that
$$\hat{\rho}_i(u')v=v+\sum_{\mu<\nu\leq\omega_i} v_\nu.$$
On the one hand, this allows us to make the following computation:

\begin{equation}
\begin{aligned}
	||\hat{\rho}_i(ak)v|| &= ||\hat{\rho}_i(k'a'u')v||=\bigg|\bigg|\hat{\rho}_i(k'a')\bigg(v+\sum_{\mu<\nu\leq\omega_i} v_\nu\bigg)\bigg|\bigg| \\
	&=\bigg|\bigg|\hat{\rho}_i(k') \bigg(a'(\mu)v+\sum_{\mu<\nu\leq\omega_i}a'(\nu)v_\nu\bigg)\bigg|\bigg| \\
	&=\bigg|\bigg|a'(\mu)v+\sum_{\mu<\nu\leq\omega_i}a'(\nu)v_\nu\bigg|\bigg| \\
	&\geq ||a'(\mu)v||=a'(\mu)=a'(w^{-1}(\omega_i))=w(a')(\omega_i)
\end{aligned} \label{RepTheoryInequality}
\end{equation}

\noindent To obtain the inequality \eqref{RepTheoryInequality}, we have used that the $K$-action is unitary, that weight vectors in different weight spaces are orthogonal, $||v||=1$ and that $a'(\mu)=e^{\langle\mu,h'\rangle}>0$. On the other hand, the vector $\hat{\rho}_i(ak)v$ is contained in the image of the unit sphere under the diagonal map $\hat{\rho}_i(a)$. More precisely, since $L(\omega_i)$ has a basis consisting of weight vectors and $\hat{\rho}_i(k)v$ is a finite linear combination of basis vectors, there exists some $m\in\N$ such that
$$\hat{\rho}_i(k)v\in \bigoplus_{\substack{\alpha\in Q_+, \\ \op{ht}(\alpha)\leq m}} L(\omega_i)_{\omega_i-\alpha}=:V_0,$$
which is a finite-dimensional vector space, and $\hat{\rho}_i(k)v$ is contained in the unit sphere of $V_0$. Since $\hat{\rho}_i(a)$ acts diagonally on $L(\omega_i)$, it preserves $V_0$ and the restriction of $\hat{\rho}_i(a)$ to $V_0$ is represented by a finite diagonal matrix whose diagonal entries are the values $a(\omega_i-\alpha)$, where $\alpha$ runs over all elements of $Q_+$ such that $\op{ht}(\alpha)\leq m$ and $\omega_i-\alpha\in\op{wt}(L(\omega_i))$. Hence, $\hat{\rho}_i(ak)v$ is contained in the image of the unit sphere of $V_0$ under the diagonal map $\hat{\rho}_i(a)|_{V_0}$. This set is an ellipsoid whose half-axes have lengths given by the absolute values of the eigenvalues of $\hat{\rho}_i(a)|_{V_0}$. Since these are all positive, the absolute values can be omitted. In particular, the longest half-axis is the maximal eigenvalue $a(\omega_i)$ and we conclude that
$$w(a')(\omega_i)\leq ||\hat{\rho}_i(ak)v||\leq a(\omega_i).$$
This is equivalent to $e^{\langle \omega_i,w(h')\rangle}\leq e^{\langle \omega_i,h\rangle}$ and hence to $\langle \omega_i,w(h')\rangle\leq \langle \omega_i,h\rangle$ since the exponents are real numbers. The proof of the inclusion $\Pi_a\subset\exp(\Xi_h)$ is therefore finished.

\subsubsection{Step 2: Induction} \label{nonlinInduction}

We now proceed to show the converse inclusion $\exp(\Xi_h)\subset\Pi_a$ by induction on the size $n=|I|$ of the GCM $\A$. If $n=1$, then $\A=(2)$ is the Cartan matrix of the finite-dimensional simple Lie algebra $\mf{sl}_2(\K)$ and the associated Weyl group is isomorphic to $\Z_2$. By \cite[Exercise 7.33]{Marquis}, the Kac-Moody group associated to the minimal free and cofree root datum $\mc{D}^{\A}_{\text{Kac}}$ (which coincides with $\mc{D}^{\A}_{\text{sc}}$ since $\A$ is invertible) is $\SL_2(\K)$ and $G$ is a semidirect extension of that, i.e. $G\cong \SL_2(\K)\rtimes T'$ where $T'$ is a subtorus of $T$ of dimension $d-1$. In contrast to the Lie algebra case (cf. the discussion after Assumption \ref{assum:GlobalAssumption}), this is not necessarily a direct product. \\
Before we treat the case of an arbitrary free and cofree root datum, let us assume that $\mc{D}=\mc{D}^{\A}_{\text{Kac}}$. In this case, $G=\SL_2(\K)$ and the proof can be completely reduced to $(2\times 2)$-matrix computations. If we identify the positive (resp. negative) root group of $G$ with the group of upper (resp. lower) unitriangular matrices, then the compact involution of $G$ is the map $g\mapsto (\overline{g}^{-1})^T$. It induces the usual Iwasawa decomposition $G=KAU_+$, where $K=\SO(2)$ if $\K=\R$ and $K=\op{SU}(2)$ if $\K=\C$, $A$ is the group of diagonal matrices with determinant $1$ and positive real entries (both in the real and complex case) and $U_+$ is the upper unitriangular group. Hence, we have
$$h=\begin{pmatrix} r & 0 \\ 0 & -r\end{pmatrix} \text{and } a=\exp(h)=\begin{pmatrix} e^r & 0 \\ 0 & e^{-r}\end{pmatrix}  \text{for some $r\in\R$.}$$
Since the positive root of $\g=\mf{sl}_2(\K)$ is the map $\begin{psmallmatrix} r & 0 \\ 0 & -r\end{psmallmatrix}\mapsto 2r$ and we assume that $h\in C$, we have $r>0$. The unique non-trivial element $s_1\in W\cong\Z_2$ acts on the one-dimensional space $\a$ as $-\id$. We parametrize the group $\SO(2)\subset K$ as
$$\SO(2)=\bigg\{R_\gamma:=\begin{pmatrix} \cos(\gamma) & -\sin(\gamma) \\ \sin(\gamma) & \cos(\gamma) \end{pmatrix}\,\bigg|\, \gamma\in [0,2\pi)\bigg\}.$$

\begin{prop} \label{prop:SL2Kostant}
	For every $\gamma\in [0,2\pi)$, we have
	$$\Pi(a R_\gamma)=\begin{pmatrix} \sqrt{e^{2r}\cos^2(\gamma)+e^{-2r}\sin^2(\gamma)} & 0 \\ 0 & \frac{1}{\sqrt{e^{2r}\cos^2(\gamma)+e^{-2r}\sin^2(\gamma)}}\end{pmatrix}.$$
\end{prop}

\begin{proof}
	Throughout this proof, we denote by $(\cdot,\cdot)$ the standard symmetric bilinear or Hermitian form on $\K^2$ and by $||\cdot ||$ the induced norm. To compute the Iwasawa decomposition of $aR_\gamma$, we need to apply the Gram-Schmidt procedure to the column vectors of the matrix
	$$aR_\gamma=\begin{pmatrix} e^r\cos(\gamma) & -e^r\sin(\gamma) \\ e^{-r}\sin(\gamma) & e^{-r}\cos(\gamma) \end{pmatrix}.$$
	Let $\{v_1,v_2\}$ denote the two column vectors of this matrix, then the Gram-Schmidt process turns $\{v_1,v_2\}$ into an orthonormal basis $\{e_1,e_2\}$ of $\K^2$ as follows: First, set $\tilde{e}_1:=v_1$ and
	\begin{align*}
		\tilde{e}_2 &:= v_2-\frac{(v_1,v_2)}{||v_1||^2} v_1=\begin{pmatrix} -e^r\sin(\gamma) \\ e^{-r}\cos(\gamma)\end{pmatrix}-\frac{(e^{-2r}-e^{2r})\cos(\gamma)\sin(\gamma)}{e^{2r}\cos^2(\gamma)+e^{-2r}\sin^2(\gamma)}\begin{pmatrix} e^{r}\cos(\gamma) \\ e^{-r}\sin(\gamma) \end{pmatrix} \\
		&=\frac{1}{e^{2r}\cos^2(\gamma)+e^{-2r}\sin^2(\gamma)}\bigg[\begin{pmatrix} -e^r\sin(\gamma)[e^{2r}\cos^2(\gamma)+e^{-2r}\sin^2(\gamma)] \\ e^{-r}\cos(\gamma)[e^{2r}\cos^2(\gamma)+e^{-2r}\sin^2(\gamma)] \end{pmatrix} \\
		&\quad -\begin{pmatrix} (e^{-2r}-e^{2r})\cos(\gamma)\sin(\gamma)\cdot e^{r}\cos(\gamma) \\ (e^{-2r}-e^{2r})\cos(\gamma)\sin(\gamma)\cdot e^{-r}\sin(\gamma) \end{pmatrix}\bigg] =\frac{1}{e^{2r}\cos^2(\gamma)+e^{-2r}\sin^2(\gamma)}\begin{pmatrix} -e^{-r}\sin(\gamma) \\ e^r \cos(\gamma)\end{pmatrix}.
	\end{align*}
	Then $e_1$ and $e_2$ are obtained by normalizing $\tilde{e}_1$ and $\tilde{e}_2$ and we have
	$$||\tilde{e}_1||=\sqrt{e^{2r}\cos^2(\gamma)+e^{-2r}\sin^2(\gamma)} \qquad \text{and} \qquad ||\tilde{e}_2||=\frac{1}{\sqrt{e^{2r}\cos^2(\gamma)+e^{-2r}\sin^2(\gamma)}}.$$
	Since column operations are realized by multiplication with elementary matrices from the right, it follows that $k'=gu'^{-1}a'^{-1}$ where: 
	\begin{equation}
		\begin{aligned}
			k'&=\frac{1}{\sqrt{e^{2r}\cos^2(\gamma)+e^{-2r}\sin^2(\gamma)}}\begin{pmatrix} e^{r}\cos(\gamma) & -e^{-r}\sin(\gamma) \\ e^{-r}\sin(\gamma) & e^{r}\cos(\gamma) \end{pmatrix}\in \SO(2)\subset K \\
			a'^{-1}&=\begin{pmatrix} \frac{1}{\sqrt{e^{2r}\cos^2(\gamma)+e^{-2r}\sin^2(\gamma)}} & 0 \\ 0 & \sqrt{e^{2r}\cos^2(\gamma)+e^{-2r}\sin^2(\gamma)}\end{pmatrix}\in A \\
			u'^{-1}&=\begin{pmatrix} 1 & -\frac{(e^{-2r}-e^{2r})\cos(\gamma)\sin(\gamma)}{e^{2r}\cos^2(\gamma)+e^{-2r}\sin^2(\gamma)} \\ 0 & 1\end{pmatrix}\in U_+
		\end{aligned} \label{2d_iwasawa}
	\end{equation}
	If we rearrange this expression as $g=k'a'u'$, we conclude that
	$$\Pi(aR_\gamma)=a'=\begin{pmatrix} \sqrt{e^{2r}\cos^2(\gamma)+e^{-2r}\sin^2(\gamma)} & 0 \\ 0 & \frac{1}{\sqrt{e^{2r}\cos^2(\gamma)+e^{-2r}\sin^2(\gamma)}}\end{pmatrix}.$$
\end{proof}

\noindent As we let $\gamma$ run through the interval $[0,2\pi)$, the expression $\sqrt{e^{2r}\cos^2(\gamma)+e^{-2r}\sin^2(\gamma)}$ attains precisely the values in the interval $[e^{-r},e^r]\subset\R_{>0}$, which is in bijective correspondence with $[-r,r]\subset\R$ under the exponential map. Therefore, $\Pi(a R_\gamma)$ takes precisely the values in the exponential of the line segment connecting  
$$h=\begin{pmatrix} r & 0 \\ 0 & -r\end{pmatrix} \qquad \text{with} \qquad s_1(h)=\begin{pmatrix} -r & 0 \\ 0 & r\end{pmatrix},$$
which is the convex hull of the $W$-orbit of $h$. This shows that $\exp(\Xi_h)\subset \Pi(a\SO(2))\subset\Pi(aK)=\Pi_a$. These inclusions are automatically equalities if $\K=\R$, but they are also equalities if $\K=\C$ by Step 1 of the proof. In particular, we see that $\Pi(a\SO(2))=\Pi(a\op{SU(2)})$ in the complex case. 
\medskip \newline
Let now $\mc{D}$ be an arbitrary free and cofree Kac-Moody root datum associated to $\A=(2)$ and let $G$ be the corresponding Kac-Moody group of type $\mc{D}$ over $\K$. Then the subgroup $G_1$ of $G$ generated by $U_{\pm\alpha_1}$ and $T_1:=\{r^{\alpha_1^\vee}\mid r\in\K^\ast\}$ is isomorphic to $\SL_2(\K)$ under the map $\varphi:\SL_2(\K)\to G$ defined by $\varphi\begin{psmallmatrix} 1 & r \\ 0 & 1 \end{psmallmatrix}:=x_{\alpha_1}(r)$, $\varphi\begin{psmallmatrix} 1 & 0 \\ r & 1 \end{psmallmatrix}:=x_{-\alpha_1}(-r)$ for $r\in\K$ and $\varphi\begin{psmallmatrix} r & 0 \\ 0 & r^{-1} \end{psmallmatrix}:=r^{\alpha_1^\vee}$ for $r\in\K^\ast$. If $T_2$ is any complementary subgroup to $T_1$ in $T$, then $G\cong G_1\rtimes T_2$, where the conjugation action of $T_2$ is trivial on $T_1$ and satisfies $tx_{\pm \alpha_1}(r)t^{-1}=x_{\pm\alpha_1}(t(\pm \alpha_1)r)$ for all $t\in T_2$ and $r\in\K$. In particular, if $t(\alpha_1)=1$, then $t$ commutes with $G_1$ and thus lies in the center of $G$. The compact involution of $G$ preserves $G_1$ and under the isomorphism $\varphi$, it corresponds to the usual compact involution $g\mapsto (\overline{g}^{-1})^T$ of $\SL_2(\K)$. Hence, $G_1$ inherits an Iwasawa decomposition from $G$ by intersection, i.e. $G_1=K_1A_1U_1$, where $K_1=K\cap G_1$ is isomorphic to $\SO(2)$ if $\K=\R$ and to $\op{SU}(2)$ if $\K=\C$, $A_1=A\cap G_1=\{r^{\alpha_1^\vee}\mid r\in\R_{>0}\}$ and $U_1=U_+\cap G_1=U_{\alpha_1}$, and these three groups are simply the images under $\varphi$ of the corresponding subgroups in the usual Iwasawa decomposition of $\SL_2(\K)$ described above. \\
In order to deduce the statement $\exp(\Xi_h)\subset\Pi_a$ for $a\in A\subset G$ from that of $\SL_2(\K)$, we claim that there exist $a_1,a_2\in A$ such that $a=a_1a_2$ with $a_1\in A_1$ and $a_2$ in the center of $G$. Indeed, the Cartan subalgebra $\a$ is a direct sum $\a=\a_1\oplus\a_2$, where $\a_1=\R\alpha_1^\vee$ and $\a_2$ is an arbitrary complement, so $h=h_1'+h_2'$ with $h_1'\in\a_1$ and $h_2'\in\a_2$. Set $r:=\alpha_1(h_2')\in\R$ and rewrite $h$ as $h=h_1+h_2$, where $h_1:=h_1'+\frac{r}{2}\alpha_1^\vee\in\a_1$ and $h_2:=h_2'-\frac{r}{2}\alpha_1^\vee$, so that $\alpha_1(h_2)=0$. The exponential $a_2:=\exp(h_2)\in A$ then satisfies $a_2(\alpha_1)=1$, which implies that $a_2$ is central in $G$. In addition, we have $a_1:=\exp(h_1)\in A_1$ and $a=\exp(h)=\exp(h_1+h_2)=a_1a_2$, as claimed. If we choose such a decomposition for our fixed element $a\in A$, then we have $\Pi(ak)=\Pi(a_1k)a_2$ for every $k\in K$. Since we have already shown the result for $\SL_2(\K)$ and $\varphi$ is compatible with the Iwasawa decompositions in $\SL_2(\K)$ and $G_1$, it follows that
\begin{align*}
	\Pi_a\supset\Pi(aK_1) &= \Pi(a_1K_1)a_2\supset\exp(\conv(W\cdot h_1))a_2=\exp(\conv(W\cdot h_1)+h_2) \\
	&=\exp(\conv(W\cdot (h_1+h_2)))=\exp(\conv(W\cdot h))=\exp(\Xi_h).
\end{align*}
Here, the first equality in the second line is justified by the fact that $W\cong \Z_2$ fixes $h_2$ since $\alpha_1(h_2)=0$. Thus, the proof of the case $n=1$ is finished.
\medskip \newline
Assume from now on that $n>1$ is arbitrary. To prepare for the induction step, we need to address a technical issue first. Recall from \eqref{LeviDecomp} and the discussion following it that for any subset $J\subset I$, the parabolic subgroup $P_J\subset G$ has a Levi decomposition $P_J\cong G_J\ltimes U_+^\perp(J)$, where $G_J$ is a Kac-Moody group associated to a free and cofree root datum $\mc{D}(J)$ of rank $d$ with GCM $\A_J$. Hence, $G_J$ has an Iwasawa decomposition $G_J=K_JAU_+(J)$ and a corresponding projection $\Pi_J:G_J\to A$, where $K_J$ is the fixed-point set of the compact involution $\Theta_J:G_J\to G_J$ and $U_+(J)$ is the group generated by all $U_\alpha$ with $\alpha\in \Phi_+(J)$. Observe that the group $A$ has not changed since the root data $\mc{D}$ and $\mc{D}(J)$ have the same free $\Z$-module $\Lambda$. By definition, $\Theta_J$ is simply the restriction of the compact involution $\Theta:G\to G$ to $G_J$, so that $K_J\subset K$ and we also have $U_+(J)\subset U_+$. Hence, if we write an element $g_0\in G_J$ according to its Iwasawa decomposition in $G_J$ as $g_0=k_0a_0u_0$, then $k_0\in K$, $a_0\in A$ and $u_0\in U_+$, so uniqueness implies that this is also the Iwasawa decomposition of $g_0$ in the ambient group $G$ and $\Pi_J=\Pi|_{G_J}$. Conversely, it follows that $G_J$ contains the Iwasawa components of all its elements. The Weyl group of $\A_J$ is the standard subgroup $W_J\subset W$ and we may define an open fundamental chamber $C_J$ and a Tits cone $X_J$ for $W_J$ in $\a$ by
$$C_J:=\{h'\in\a: \langle \alpha_j,h' \rangle>0\,\, \forall j\in J\} \qquad \text{and} \qquad X_J:=\bigcup_{w\in W_J} w(\overline{C_J}).$$
If $J\subsetneq I$, the induction hypothesis applied to $G_J$ yields that for every $a_0=\exp(h_0)\in A$, where $h_0$ is contained in an open chamber of $X_J\subset\a$, we have
	\begin{equation}
		\Pi(a_0K_J)=\Pi_J(a_0K_J)\supset\exp(\conv(W_J\cdot h_0)). \label{InductionPrep1}
	\end{equation}
In case that $J=I\setminus\{i\}$ for some $i\in I$, we set $\A^{(i)}:=\A_{I\setminus\{i\}}$, $C^{(i)}:=C_{I\setminus\{i\}}$ and $X^{(i)}:=X_{I\setminus\{i\}}$ and write the Iwasawa decomposition of $G^{(i)}:=G_{I\setminus\{i\}}$ as $G^{(i)}=K^{(i)}AU_+^{(i)}$. \\
More generally, let $w\in W$ be arbitrary and choose $\tilde{w}\in\tilde{W}\subset K$ such that conjugation by $\tilde{w}$ induces the action of $w$ on $A$ and $\Phi$. The subgroup $G_{w(J)}:=\tilde{w}G_J\tilde{w}^{-1}\subset G$ is isomorphic to $G_J$ and inherits an Iwasawa decomposition $G_{w(J)}=K_{w(J)}AU_+(w(J))$, where $K_{w(J)}=\tilde{w}K_J\tilde{w}^{-1}$ and $U_+(w(J))=\tilde{w}U_+(J)\tilde{w}^{-1}$. Here, we have $K_{w(J)}\subset K$ and $U_+(w(J))$ is the group generated by all $U_{w(\alpha)}$ with $\alpha\in\Phi_+(J)$, but it is not necessarily true that $U_+(w(J))\subset U_+$ since some root $w(\alpha)$ might be negative. However, the definition of $G_{w(J)}$ is independent of the left coset $wW_J\subset W$, so we may assume that $w$ is the unique element of minimal length in $wW_J$. Then $l(w)<l(ws_j)$ for all $j\in J$, which implies that $w(\alpha_j)\in\Phi_+$ for all $j\in J$ by \eqref{PositivityCriterion} and hence $w(\alpha)\in\Phi_+$ for all $\alpha\in\Phi_+(J)$. With this choice, we have $U_+(w(J))\subset U_+$, so that the Iwasawa decomposition of $G_{w(J)}$ is simply the restriction of that of $G$. In particular, $G_{w(J)}$ contains the Iwasawa components of all its elements and the projection $\Pi_{w(J)}:G_{w(J)}\to A$ satisfies $\Pi_{w(J)}=\Pi|_{G_{w(J)}}$ and $\Pi_{w(J)}(\tilde{w}g_0\tilde{w}^{-1})=\tilde{w}\Pi_J(g_0)\tilde{w}^{-1}$ for all $g_0\in G_J$. If $J\subsetneq I$, we conclude from \eqref{InductionPrep1} that for every $a_0=\exp(h_0)\in A$ such that $h_0$ lies in an open chamber of $X_J$, we have
	\begin{equation}
		\begin{aligned}
			\Pi(a_0K_{w(J)}) &= \Pi_{w(J)}(a_0K_{w(J)})=\Pi_{w(J)}(a_0\tilde{w}K_J\tilde{w}^{-1}) =\Pi_{w(J)}(\tilde{w}w^{-1}(a_0)K_J\tilde{w}^{-1}) \\
			&=\tilde{w}\Pi_J(w^{-1}(a_0)K_J)\tilde{w}^{-1}\supset \tilde{w}\exp(\conv(W_J\cdot w^{-1}(h_0)))\tilde{w}^{-1} \\
			&=\exp(w(\conv(W_J\cdot w^{-1}(h_0))))=\exp(\conv(wW_Jw^{-1}\cdot h_0)).
		\end{aligned} \label{InductionPrep2}
	\end{equation}
We now continue with the proof of $\exp(\Xi_h)\subset\Pi_a$. Our strategy is to ``sweep out'' $\Xi_h$ by parallel translates of hyperplanes to which we may apply the induction hypothesis. We proceed as in Section \ref{regularOrbit}. Let $i\in I$ be arbitrary and consider the hyperplanes $H^{(i)}(t):=\{y\in E_h: \langle \omega_i,y\rangle=t\}$ for $t\in\R$. We have seen in \eqref{XiFoliation} that $\Xi_h=\bigcup_{t\in I_h} P^{(i)}(t)$, where $P^{(i)}(t)=\Xi_h\cap H^{(i)}(t)$ and $I_h\subset\R$ is the interval for which this intersection is not empty. Thus, it suffices to show that $\exp(P^{(i)}(t))\subset\Pi_a$ for every $t\in I_h$. \\
Fix an arbitrary $t\in I_h$. First, we are going to show that for every face $F'\subset P^{(i)}_{\op{ess}}(t)$, we have $\exp(F')\subset\Pi_a$. Observe that if $F'$ is contained in the essential part, it has to be a proper face of $P^{(i)}(t)$, which implies $\op{dim}(F')<\op{dim}(P^{(i)}(t))=n-1$. By \eqref{FaceIntersection}, there exists a unique face $F$ of $\Xi_h$ such that $F'=F\cap H^{(i)}(t)$ and it satisfies $\op{dim}(F)\leq\op{dim}(F')+1<n$, so $F$ is a proper face of $\Xi_h$. By Theorem \ref{thm:dualcoxeter_iso}, there is a unique proper subset $J\subsetneq I$ and some $w\in W$ such that 
	$$F=w(F_J)=\op{conv}(wW_J\cdot h)=\op{conv}(wW_Jw^{-1}\cdot w(h)).$$
Applying \eqref{InductionPrep2} to the subgroup $G_{w(J)}\subset G$ and the point $a_0=w(a)$ yields that
	$$\exp(F')\subset\exp(F)=\exp(\op{conv}(wW_Jw^{-1}\cdot w(h)))\subset \Pi(w(a)K_{w(J)})=\Pi(aK_J\tilde{w}^{-1})\subset\Pi_a,$$
where $\tilde{w}\in\tilde{W}\subset K$ is a representative of $w$ in the extended Weyl group. Since the face $F'$ was arbitrary, we have shown that $\exp(P^{(i)}_{\op{ess}}(t))\subset\Pi_a$. For the next step, however, the more precise statement $\exp(F')\subset \Pi(w(a)K_{w(J)})$ will be important. \\	
Let now $a''\in\exp(P^{(i)}(t))$ be arbitrary and set $h'':=\log(a'')\in P^{(i)}(t)$. By Proposition \ref{prop:essentialHull}, there exists a point $h'\in P^{(i)}_{\op{ess}}(t)$ such that $h''\in\conv(W^{(i)}\cdot h')$. Let $F'\subset P^{(i)}_{\op{ess}}(t)$ be a face of the essential part containing $h'$ and set $a':=\exp(h')$. By what we have just discussed, there exist $w\in W$ and a proper subset $J\subsetneq I$ such that $F'\subset F=w(F_J)$ and $a'\in\Pi(w(a)K_{w(J)})$. Hence, we find $k_1,k'\in K_{w(J)}$ and $u'\in U_+(w(J))$ such that
\begin{equation}
	w(a)k_1=k'a'u'. \label{Iwasawa1}
\end{equation}
Now let $P^{(i)}:=P_{I\setminus\{i\}}$ be the maximal parabolic subgroup of type $I\setminus\{i\}$ with Levi decomposition $P^{(i)}\cong G^{(i)}\ltimes U_+^\perp(I\setminus\{i\}),$ where the Levi factor $G^{(i)}$ is generated by $T$ and all root groups $U_\alpha$ with $\alpha\in\Phi^{(i)}:=\Phi(I\setminus\{i\})$ and the unipotent radical $U_+^\perp(I\setminus\{i\})$ is generated by all $U_\alpha$ with $\alpha\in\Phi_+\setminus\Phi^{(i)}$. As $h'\in P^{(i)}_{\op{ess}}(t)\subset C^{(i)}$ is contained in an open chamber of $X^{(i)}$, we may apply \eqref{InductionPrep2} to the subgroup $G^{(i)}\subset G$ and the point $a'\in A$ to conclude that $\exp(\conv(W^{(i)}\cdot h'))\subset\Pi(a'K^{(i)})$. Thus, we find $k_2,k''\in K^{(i)}$ and $u''\in U_+^{(i)}$ such that
\begin{equation}
	a'k_2=k''a''u''. \label{Iwasawa2}
\end{equation}
Combining \eqref{Iwasawa1} and \eqref{Iwasawa2} leads to
\begin{equation}
	w(a)k_1k_2=k'a'u'k_2=k'a'k_2(k_2^{-1}u'k_2)=k'k''a''u''(k_2^{-1}u'k_2). \label{Iwasawa3}
\end{equation}
We claim that $k_2^{-1}u'k_2\in U_+$. In general, it is rarely the case that a $K$-conjugate of some element in $U_+$ also lies in $U_+$. In our situation, this is a consequence of our construction. \\
Since $u'\in U_+(w(J))$, we can write it as a product $u'=u_1\ldots u_s$ such that for every $r\in\{1,\ldots,s\}$, $u_r$ is an element of a root group $U_{\beta_r}$, where $\beta_r\in\Phi_+$ and $\gamma_r:=w^{-1}(\beta_r)\in\Phi(J)$. We claim that $\beta_r\in\Phi_+\setminus\Phi^{(i)}$ for every $r\in\{1,\ldots,s\}$. Assume for a contradiction that $\beta_r\in\Phi^{(i)}$ for some $r\in\{1,\ldots,s\}$, then the reflection $s_{\beta_r}=ws_{\gamma_r}w^{-1}$ would be an element of $W^{(i)}$. In addition, $s_{\gamma_r}\in W_J$ stabilizes $F_J=\conv(W_J\cdot h)$, so $s_{\beta_r}$ stabilizes $w(F_J)=\conv(wW_Jw^{-1}\cdot w(h))=F$. Hence, $s_{\beta_r}$ stabilizes $F\cap H^{(i)}(t)=F'$. This is a contradiction because $F'$ is contained in the essential part of $P^{(i)}(t)$, which is a subset of the open fundamental chamber $C^{(i)}$ of $X^{(i)}$, and is therefore not invariant under any non-trivial element of $W^{(i)}$. It follows that $\beta_r\in\Phi_+\setminus\Phi^{(i)}$, as claimed. This implies that for every $r\in\{1,\ldots,s\}$, $u_r$ is contained in the unipotent radical of the maximal parabolic subgroup $P^{(i)}\subset G$. Hence, $u'$ is also contained in the unipotent radical, while $k_2\in K^{(i)}\subset G^{(i)}$ is contained in the Levi factor of $P^{(i)}$. Thus, $k_2^{-1}u'k_2$ is also contained in the unipotent radical of $P^{(i)}$ and hence in $U_+$. This shows that the expression on the right-hand side of \eqref{Iwasawa3} is indeed the Iwasawa decomposition of $w(a)k_1k_2$ in $G$ and we obtain that $a''=\Pi(w(a)k_1k_2)\subset\Pi(w(a)K)=\Pi_a$. Since $a''\in\exp(P^{(i)}(t))$ was arbitrary, we conclude that $\exp(P^{(i)}(t))\subset\Pi_a$ for all $t\in I_h$, which finishes the proof of Theorem \ref{thm:KMKostant}.

\subsection{A linear analogue} \label{LinearAnalogue}

Theorem \ref{thm:KMKostant} is ``non-linear'' in the sense that it deals with the Iwasawa decomposition on the group level, but using very similar methods, it is possible to prove a ''linear'' analogue on the level of Kac-Moody algebras. A variant of this result has already been established by V.G. Kac and D.H. Peterson in \cite[Theorem 2 b)]{KP2}. Their proof relies on transferring concepts from symplectic geometry, in particular the \textit{moment map}, to highest-weight modules of Kac-Moody algebras. Our approach allows to show the following result. \\
Let $\mc{D}$ be a free and cofree Kac-Moody root datum associated to a symmetrizable GCM $\A$ and let $\g$ be the real or complex Kac-Moody algebra of type $\mc{D}$. The same arguments as in the well-known finite-dimensional case show that $\g$ has an Iwasawa decomposition $\g=\k\oplus\a\oplus\mf{u}_+$, where $\k$ is the fixed-point set of the compact involution of $\g$. We denote by $\pi:\g\to\a$ the corresponding linear projection. The group $G$ acts on $\g$ via the adjoint representation $\Ad:G\to \GL(\g)$ and the following ``linear'' convexity theorem holds.

\begin{thm} \label{thm:linear_kostant}
	If $h\in\a$ is contained in an open chamber of the Tits cone $X\subset\a$, then
	$$\{\pi(\Ad(k)h)\mid k\in K\}=\conv(W\cdot h).$$
\end{thm}

\noindent The result in \cite[Theorem 2 b)]{KP2} is less general in the sense that it only considers the ``classical'' Kac-Moody algebra $\g(\A)$ and is therefore restricted to one specific root datum, but in contrast to Theorem \ref{thm:linear_kostant}, it only requires $h$ to be contained in the Tits cone $X\subset\a$ rather than an open chamber. Although Theorem \ref{thm:KMKostant} may appear to be an ``exponentiated version'' of Theorem \ref{thm:linear_kostant}, the authors of the present article are not aware whether one of the results can be deduced from the other. One major obstruction to a naive approach is that for $k\in K$ one typically has
	\begin{equation}
		\Pi\circ\op{conj}_k\circ\exp|_{\a}\neq \exp\circ\pi\circ\Ad(k)|_{\a}, \label{ProjectionDifference}
	\end{equation}
where $\op{conj}_k$ denotes conjugation by $k$ in $G$. A counterexample is already easy to find in the most basic example $G=\SL_2(\R)$. Only when an element $h$ in an open chamber of $X$ is fixed and $k$ runs through the entire group $K$ do the images of the resulting two maps coincide. A direct proof of this assertion would be a very interesting achievement that might lead to various generalizations of Theorem \ref{thm:KMKostant}. In the case where $G$ is a connected semisimple Lie group, such a ``linear'' convexity theorem has already been established by Kostant in \cite[Theorem 8.2]{Kostant} and a direct relation between the linear and the non-linear result was found by J.J. Duistermaat in \cite[Theorem 1.1]{Duistermaat} by constructing diffeomorphisms of $K$ that resolve the inequality in \eqref{ProjectionDifference}. It relies heavily on analytic properties of Lie groups and therefore does not generalize to the Kac-Moody setting in a straightforward way.
\newpage
\noindent Although a direct connection between the linear and non-linear convexity theorem for Kac-Moody groups does not seem to be readily available, it is possible to prove Theorem \ref{thm:linear_kostant} with almost the same methods as Theorem \ref{thm:KMKostant} and we briefly describe how this can be achieved. Using the extended Weyl group $\tilde{W}\subset K$, it is easy to show that the set $\pi_h$ on the left-hand side in Theorem \ref{thm:linear_kostant} does not depend on the $W$-translate of $h$, so we may again w.l.o.g. assume that $h\in C$. Set $\Xi_h:=\conv(W\cdot h)$, let $k\in K$ be arbitrary and write
	\begin{equation}
		\Ad(k)h=X'+h'+Y' \qquad \text{with $X'\in\k$, $h'\in\a$ and $Y'\in\mf{u}_+$.} \label{RepDecomp1}
	\end{equation}
In order to show that $h'=\pi(\Ad(k)h)\in\pi_h$ is contained in $\Xi_h$, it suffices again to verify the inequalities in Corollary \ref{cor:inequality1}, i.e. $\langle\omega_i,w(h')\rangle\leq\langle\omega_i,h\rangle$ for all $w\in W$ and $i\in I$ and $\langle\phi_i,h'\rangle=\langle\phi_i,h\rangle$ for all $i\in\{n+1,\ldots,d\}$. These assertions follow as before by comparing the actions of $h$ and $h'$ in the highest-weight modules $L(\omega_i)$ and $L(\phi_i)$ endowed with a positive definite, contravariant, symmetric bilinear or Hermitian form $H$. A crucial ingredient is the formula
	\begin{equation}
		\rho_i(\Ad(k)h)=\hat{\rho}_i(k)\rho_i(h)\hat{\rho}_i(k)^{-1}, \label{RepDecomp2}
	\end{equation}
which holds more generally for integrable modules of Kac-Moody algebras. One then compares the two different ways of computing the action of $\Ad(k)h$ under $\rho_i$ provided by \eqref{RepDecomp1} and \eqref{RepDecomp2}. The equation $\langle\phi_i,h'\rangle=\langle\phi_i,h\rangle$ follows from the fact that $L(\phi_i)$ is one-dimensional, so \eqref{RepDecomp2} reduces to $\rho_i(h)$, and the observation that $\rho_i(X')=\rho_i(Y')=0$. To prove the inequalities $\langle\omega_i,w(h')\rangle\leq\langle\omega_i,h\rangle$ for $w\in W$ and $i\in I$, one studies the action of $\Ad(k)h$ on a normalized weight-$\mu$ vector $v\in L(\omega_i)_\mu$, where $\mu:=w^{-1}(\omega_i)$. On the one hand, the same geometric argument as in the non-linear case can be applied to the action given by \eqref{RepDecomp2} to obtain an upper bound for the norm of $\rho_i(\Ad(k)h)v$, but one has to take care of the possibility that some eigenvalues of $\rho_i(h)$ might be negative. This issue is easily dealt with by adding a suitable scalar multiple of the identity to $\rho_i(h)$, so that all eigenvalues become non-negative. On the other hand, it is not difficult to show that $H(v,\rho_i(X')v)=H(v,\rho_i(Y')v)=0$ which allows to compare the value of $H(\rho_i(\Ad(k)h)v,v)$ obtained from the two equations. These ideas prove that $\pi_h\subset\Xi_h$. \\
The converse inclusion follows again by induction on $n$, reducing the case $n=1$ to $\g=\mf{sl}_2(\K)$ and its usual Iwasawa decomposition by a similar argument, where the assertion is a straightforward computation. One then chooses an arbitrary $i\in I$ and studies the adjoint actions of the same subgroups $G_{w(J)}$ and $G^{(i)}$ on $h$. The geometric setup of the proof is the same as in the non-linear case, i.e. for fixed $t\in I_h$, one first covers faces of the essential part $P^{(i)}_{\op{ess}}(t)$ using the adjoint action of some subgroup $K_{w(J)}\subset G_{w(J)}$ and then uses the action of $K^{(i)}\subset G^{(i)}$ to cover the remaining faces of $P^{(i)}(t)$. Again, the crucial step is to ensure that the repeated applications of these subgroups respect the Iwasawa decompositions. More precisely, one needs to show that if $k_1\in K_{w(J)}$ and $Y$ denotes the $\mf{u}_+$-component of $\Ad(k_1)h$, then $\Ad(k_2)Y\in\mf{u}_+$ for every $k_2\in K^{(i)}$. Here, the proof does not require a geometric argument, the assertion is a consequence of \eqref{RootGroupAction} which holds in a very similar form for the adjoint representation. These arguments establish that $\Xi_h\subset\pi_h$ and conclude the proof of Theorem \ref{thm:linear_kostant}.

\subsection{Concluding remarks} \label{ConcludingRemarks}

We close our discussion by relating Theorem \ref{thm:KMKostant} to Kostant's original convexity theorem \cite[Theorem 4.1]{Kostant} and describing an application of our result to decompositions of Kac-Moody groups as well as some generalizations and open problems. 
\medskip \newline
If every indecomposable submatrix of $\A$ is of finite type, then $\A$ is the Cartan matrix of a finite-dimensional complex semisimple Lie algebra and the Kac-Moody groups associated to $\A$ essentially coincide with the corresponding Chevalley groups. In that case, $\A$ is invertible, so that $\mc{D}^{\A}_{\text{Kac}}=\mc{D}^{\A}_{\text{sc}}$, and symmetrizable. If $\mc{D}=\mc{D}^{\A}_{\text{Kac}}$, then $\g_\C$ is the complex semisimple Lie algebra with Cartan matrix $\A$ and $\g_\R$ is its split real form. Likewise, for $\K\in\{\R,\C\}$, $G=G(\K)$ is the group of $\K$-points of the corresponding simply connected Chevalley-Demazure group scheme associated to $\A$. For an arbitrary free and cofree root datum of rank $d\geq 2n-\op{rank}(\A)=n$, one obtains semidirect extensions of these groups and algebras. We can therefore view $G$ as the group of $\K$-points of a connected reductive $\R$-split algebraic $\R$-group and hence as a split real or complex Lie group with a reductive Lie algebra and finitely-many connected components in its Lie group topology. The existence of an Iwasawa decomposition for connected semisimple Lie groups is a classical result in Lie theory and has been generalized to groups of $\R$-points of connected reductive algebraic $\R$-groups in \cite[Proposition 3.13]{PR}. We also have $X=\a$, so every element is either regular or singular. Theorem \ref{thm:KMKostant} differs from \cite[Theorem 4.1]{Kostant} in the following two aspects:
\begin{itemize}
	\item[(i)] On the one hand, Kostant's result is a statement about arbitrary connected semisimple Lie groups. Not every such group arises as the group of $\K$-points of an algebraic $\R$-group, e.g. the universal cover of $\SL_2(\R)$. Moreover, by our choice of definition, the real Kac-Moody groups associated to GCM's of finite type are $\R$-split, which need not be the case for general semisimple Lie groups. On the other hand, our result does not require $G$ to be connected and allows the Lie algebra of $G$ to be reductive instead of semisimple.
	\item[(ii)] In contrast to Kostant's theorem, our result currently requires the assumption that $h$ is a regular element of $\a$. 
\end{itemize}
The differences in (i) are entirely due to the setup that we work in and not necessary for our proof which only involves algebraic and representation-theoretic arguments. Our method does not require a matrix realization and applies equally well to semisimple real Lie groups whose Lie algebra is not necessarily a split real form of its complexification. Such a group $G$ has similar structural and representation-theoretic properties, where the role of $T$ is played by a maximal $\R$-split torus $T'\subset G$. Then $G$ is generated by $T'$ and appropriate root groups, which need not be one-dimensional in general, indexed by the roots of the (possibly non-reduced) root system arising from the adjoint action of $T'$ on the Lie algebra of $G$. The second restriction can also be lifted since the Weyl group $W$ of a semisimple Lie group is finite and the convex hulls of $W$-orbits of points in $\a$ are compact. This has the consequence that if $h\in\a$ is an arbitrary (possibly singular) point in $\a$ and $(h_k)_{k\in\N}$ is a sequence of regular elements converging to $h$, which always exists in the Lie group case, then $(\Xi_{h_k})_{k\in\N}$ converges to $\Xi_h$ with respect to the Hausdorff distance in $\a$ by finiteness of $W$. Likewise, $(a_k)_{k\in\N}$ with $a_k:=\exp(h_k)$ converges to $a:=\exp(h)$ by continuity of the exponential map and $(\Pi_{a_k})_{k\in\N}$ converges to $\Pi_a$ with respect to the Hausdorff distance in $A$ by continuity of the Iwasawa projection $\Pi:G\to A$. Thus, Theorem \ref{thm:KMKostant} extends to singular elements if $W$ is finite by approximation. Alternatively, it is also possible to repeat our proof of Theorem \ref{thm:KMKostant} for singular elements. The crucial step is to generalize Corollary \ref{cor:inequality1} and Proposition \ref{prop:essentialHull} to singular points. In fact, that $\Xi_h$ is defined by the same inequalities as in the regular case has already been shown by Kostant in \cite[Lemma 3.3]{Kostant} and also follows from our discussion since $\Xi_h$ is automatically closed if $W$ is finite. The validity of Proposition \ref{prop:essentialHull} in the singular case can be obtained by approximating $h$ by regular elements as above. Thus, our proof of Theorem \ref{thm:KMKostant} does indeed provide an alternative proof of Kostant convexity in connected semisimple Lie groups as well. However, we strongly emphasize that such an approximation technique does not work in the general Kac-Moody setting because if $W$ is infinite, the existence of unbounded faces causes the convex hulls of $W$-orbits of different points to have unbounded Hausdorff distance, so there is no hope of convergence in this case. In addition, contrary to Lie groups, Kac-Moody groups usually do not admit a Cartan decomposition $G=KAK$ (cf. \cite[Proposition 6.3 (iii)]{Horn}), so an analogue of \cite[Corollary 4.2]{Kostant} that describes the set $\{\Pi(gk)\mid k\in K\}$ for an arbitrary $g\in G$ does not automatically follow in our situation.
\medskip \newline
The non-existence of this result gives rise to another interesting question concerning the algebraic structure of Kac-Moody groups. Using \cite[Corollary 4.2]{Kostant}, Kostant concludes that every semisimple Lie group $G$ has a decomposition $G=KU_+K$ (cf. \cite[Theorem 5.1]{Kostant}). For the corresponding symmetric space $G/K$, this result implies that any two points of $G/K$ are contained in a common horosphere. In the general Kac-Moody setting, however, Kostant's arguments do not apply and the question whether a similar decomposition result holds for all Kac-Moody groups appears to have been open (cf. \cite[Question 6.8]{Horn}). Our main result yields a negative answer to this question and also provides a slightly more precise result in the Lie group case.

\begin{cor} \label{cor:HoroDecomp}
	Let $\A$ be a symmetrizable GCM of size $n$, $\mc{D}$ a free and cofree Kac-Moody root datum of rank $d\geq 2n-\op{rank}(\A)$ associated to $\A$, $G$ the real or complex split Kac-Moody group of type $\mc{D}$ and $G=KAU_+$ an Iwasawa decomposition of $G$. Suppose that one of the following two conditions is satisfied:
	\begin{itemize}
		\item[(i)] $d>n$.
		\item[(ii)] $d=n$ and not every indecomposable submatrix of $\A$ is of finite type.
	\end{itemize}
	Then $G\neq KU_+K$.
\end{cor}

\begin{proof}
	If an element $g\in G$ is contained in $KU_+K$, then there exists some $k\in K$ such that $gk\in KU_+$ and hence $\Pi(gk)=1$. We are going to show that in either of the above cases, we can choose $g=a\in\exp(C)\subset A$ such that this equality is impossible. By Theorem \ref{thm:KMKostant}, it suffices to show that there exists some $h\in C$ such that $0\notin \conv(W\cdot h)=\Xi_h$. \\
	Suppose first that $d>n$, which is automatically assumed to be the case if $\A$ is not invertible, and take an arbitrary $h_1\in C$. Since $d>n$, the subspace $\mf{c}:=\bigcap_{i\in I} \op{ker}(\alpha_i)\subset\a$ is at least one-dimensional. Hence, there exists $h_2\in\mf{c}$ such that $h:=h_1+h_2$ satisfies $\langle\omega_j,h\rangle<0$ for at least one $j\in I$. In addition, we still have $\langle \alpha_i,h\rangle=\langle\alpha_i,h_1\rangle>0$ for all $i\in I$ and thus $h\in C$. Corollary \ref{cor:inequality1} therefore implies that $0\notin\Xi_h$. \\
	Now assume that $d=n$ and let $\A=\A_1\oplus\ldots\oplus\A_m$ be the decomposition of $\A$ into indecomposable submatrices. By assumption, there exists some $k\in\{1,\ldots,m\}$ such that $\A_k$ is not of finite type. The classification of GCM's (cf. Section \ref{KMAlgebras}) therefore implies that $\A_k$ is of affine or indefinite type, where the former is impossible since $d=n$ forces $\A$ to be invertible. Denote the subset of $I$ indexing $\A_k$ by $I_k$. For a vector $h=\sum_{j=1}^n \lambda_j\alpha_j^\vee\in\a'=\a$, the condition $\langle\alpha_i,h\rangle>0$ for all $i\in I$ is equivalent to $\A^T\lambda>0$, where $\lambda=(\lambda_1,\ldots,\lambda_n)^T\in\R^n$, and the obvious analogue of that assertion is true if one passes to subsets of $I$ indexing the corresponding indecomposable submatrices of $\A$. Let $\mu$ be the vector obtained from $\lambda$ by removing all coordinates in $I\setminus I_k$. Since $\A_k$ is of indefinite type, the same is true for $\A_k^T$ and the condition $\A_k^T\mu\geq 0$ with $\mu\geq 0$ can only be satisfied if $\mu=0$, so if $\A_k^T\mu>0$, then at least one coordinate of $\mu$ must be negative. Thus, if $h\in C$, then some coordinate of the subvector $\mu$ must be negative, i.e. $\lambda_j<0$ for at least one $j\in I_k$ or, equivalently, $\langle\omega_j,h\rangle<0$. As in the first case, it now follows that $0\notin\Xi_h$.
\end{proof}

\noindent Returning to the discussion at the beginning of this section, the next natural step would be to generalize Theorem \ref{thm:KMKostant} to singular points also in the ``proper'' Kac-Moody setting where $W$ is infinite. This might also lead to a generalization of our result to Kac-Moody groups whose root datum is not necessarily free and cofree, because such groups also possess an associated Tits cone, but its open fundamental chamber might be empty. As mentioned, simple approximation techniques as in the finite case are not directly applicable. We believe that it is a more promising approach to try to generalize all the results from Chapter \ref{ConvexHulls} to the singular case. The main difficulty is to understand how faces of $\Xi_h$ degenerate as $h$ moves towards certain walls of the fundamental chamber, which is closely connected to studying how the stabilizer of $h$ decomposes the Coxeter diagram of $(W,S)$. Contrary to the regular case, however, it does not seem easy to decide whether $\Xi_h$ is closed if $h$ is singular. Another interesting open problem would be to investigate whether an analogue of Theorem \ref{thm:KMKostant} is also true for external elements. This is certainly a necessary step in order to decide whether the set $\{\Pi(gk)\mid k\in K\}$ for arbitrary $g\in G$ has an explicit description as in the case of Lie groups.